\documentclass[12pt, a4paper]{amsart}

\usepackage[
  hmarginratio={1:1},     
  vmarginratio={1:1},     
  textwidth=148mm,        
  textheight=220mm,       
  marginparwidth=27mm,
  marginparsep=3mm
]{geometry}

\usepackage[cp855]{inputenc}
\usepackage{lmodern}
\usepackage{url}
\usepackage{amsmath, amsthm, amssymb, amscd, accents, bm}
\usepackage{mathtools}
\usepackage{mdwlist}
\usepackage{paralist}
\usepackage{graphicx}
\usepackage{epstopdf}
\usepackage{datetime}
\usepackage{afterpage}
\usepackage[dvipsnames]{xcolor}
\usepackage{hyperref}
\usepackage{caption}
\usepackage{subcaption}
\mathtoolsset{showonlyrefs}
\usepackage[sort,nocompress]{cite}
\usepackage{marginnote}
\usepackage{comment}
\usepackage[foot]{amsaddr}
\usepackage{bm}
\usepackage{scalerel}
\usepackage[normalem]{ulem}
\usepackage{tikz}
\usetikzlibrary{positioning}

\usepackage{mathtools}
\usepackage{nicematrix}

\newcommand\rurl[1]{%
  \href{https://#1}{\nolinkurl{#1}}%
}

\mathtoolsset{showonlyrefs}

\renewcommand*{\hat}{\widehat}

\newcommand{\Sp}{{\mathrm{Sp}(4d,\R)}}

\newcommand{\Spp}{{\mathrm{Sp}(2d,\R)}}
\newcommand{\Mpp}{{\mathrm{Mp}(2d,\R)}}

\newcommand{\re}{{\mathrm{Re} \,}}
\newcommand{\im}{{\mathrm{Im} \,}}

\newcommand{\W}{\mathrm{W}}
\newcommand{\mmm}{M}
\newcommand{\J}{{\mathcal{J}}}
\newcommand{\V}{{\mathcal{V}}}
\newcommand{\D}{\mathcal{D}}
\newcommand{\Ro}{\mathcal{R}}
\newcommand{\A}{\mathcal{A}}
\newcommand{\B}{\mathcal{B}}

\newcommand{\bJ}{\bm{\mathcal{J}}}
\newcommand{\bV}{\bm{\mathcal{V}}}
\newcommand{\bD}{\bm{\mathcal{D}}}
\newcommand{\bRo}{\bm{\mathcal{R}}}
\newcommand{\bA}{{\bm{\mathcal{A}}}}

\newcommand{\ltd}{{L^2(\R^d)}}

\def\R{{\mathbb R}}
\def\T{{\mathbb T}}
\def\C{{\mathbb C}}

\newcommand{\Rd}{{\R^d}}
\newcommand{\Rk}{{\R^k}}
\newcommand{\Rdd}{{\R^{2d}}}
\newcommand{\Cd}{{\C^d}}

\newcommand{\nc}{\mathrm{nc}}

\theoremstyle{plain}
\newtheorem{definition}{Definition}[section]
\newtheorem{theorem}[definition]{Theorem}
\newtheorem{lemma}[definition]{Lemma}
\newtheorem{corollary}[definition]{Corollary}
\newtheorem{proposition}[definition]{Proposition}

\theoremstyle{definition}
\newtheorem{remark}[definition]{Remark}

\numberwithin{equation}{section}

\def\R{{\mathbb R}}
\def\T{{\mathbb T}}
\def\C{{\mathbb C}}

\newcommand{\diag}{\mathrm{diag}}

\newcommand{\F}{{\mathcal{F}}}

\newcommand{\mfr}{metaplectic \tfr }

\makeatletter
\newcommand\thankssymb[1]{\textsuperscript{\@fnsymbol{#1}}}

\makeatletter
\def\@makefnmark{%
  \leavevmode
  \raise.9ex\hbox{\fontsize\sf@size\z@\normalfont\tiny\@thefnmark}}

\makeatletter
\def\bign#1{\mathclose{\hbox{$\left#1\vbox to8.5\p@{}\right.\n@space$}}\mathopen{}}

\usepackage{xifthen}
\usepackage{eqparbox}
\usepackage{makebox}
 
\newcommand{\abs}[1]{\lvert #1\rvert}
\newcommand{\norm}[1]{\lVert #1\rVert}



\newcommand{\up}{uncertainty principle}
\newcommand{\tfa}{time-frequency analysis}
\newcommand{\tfr}{time-frequency representation}
\newcommand{\ft}{Fourier transform}
\newcommand{\stft}{short-time Fourier transform}

\newcommand{\tf}{time-frequency}

\newcommand{\tfs}{time-frequency shift}

\newcommand{\psdo}{pseudodifferential operator}

\newcommand{\beqa}{\begin{eqnarray*}}
\newcommand{\eeqa}{\end{eqnarray*}}

\newcommand{\field}[1]{\mathbb{#1}}
\newcommand{\bR}{\field{R}}        
\def\rd{\bR^d}

\def\rdd{{\bR^{2d}}}

\def\lrd{L^2(\rd)}
\def\lrdd{L^2(\rdd)}

\def\<{\left<}
\def\>{\right>}

\def\inv{^{-1}}

\def\mv1{M_v^1}

\begin{document}
\title[Uncertainty Principles for  Time-Frequency
  Representations]{More Uncertainty Principles for Metaplectic Time-Frequency
  Representations}
\address{\thankssymb{1}Faculty of Mathematics, University of Vienna, Oskar-Morgenstern-Platz 1, 1090 Vienna, Austria}

\author[\qquad \qquad \qquad \qquad \quad Karlheinz Gr\"ochenig]{Karlheinz Gr\"ochenig \thankssymb{1} 
}
\email{karlheinz.groechenig@univie.ac.at}

\author[Irina Shafkulovska\qquad\qquad\qquad ]{Irina Shafkulovska\thankssymb{1}
}
\email{irina.shafkulovska@univie.ac.at}

\thanks{I.~Shafkulovska was funded in part by the Austrian Science Fund (FWF) [\href{https://doi.org/10.55776/P33217}{10.55776/P33217}] and [\href{https://doi.org/10.55776/Y1199}{10.55776/Y1199}]. 
 For open access purposes, the authors have applied a CC BY public copyright license to any
author-accepted manuscript version arising from this submission.}
\allowdisplaybreaks
\belowdisplayshortskip0pt

\maketitle

\begin{abstract}
We develop a method for the transfer of an uncertainty principle for
the short-time Fourier transform or a Fourier pair to  
an uncertainty principle for
a sesquilinear or quadratic metaplectic time-frequency representation. In particular, we derive Beurling-type and Hardy-type uncertainty principles for metaplectic time-frequency representations. 


\noindent \textbf{81S07, 81S30, 22E46.}
\end{abstract}
\section{Introduction} 
The \up\  expresses many facets of the fact that a function and its
\ft\ cannot be small simultaneously~\cite{havin-joricke}. 
In its original
version, the \up\ is formulated as an inequality for the Fourier pair
$(f, \hat{f})$, yet in the realm of quantum mechanics and signal
processing one often prefers to use phase-space representations or joint
\tf\ representations. 
For these, the \up\ becomes an  inequality about a joint
\tf\ representation. 
Some early examples are the \up s of  Hardy,
Benedicks-Amrein-Berthier, and of Beurling  for the
\stft\ and the Wigner distribution~\cite{BonamiEtAl2003,Demange2005,GroechenigZimmermann2001,
janssen98,Wilczock98}. 
In a further extension, one may try to formulate
any type of \up\  for any  \tfr. 
This endeavor may easily get out of
hand, unless one restricts the attention to an
interesting and natural class of \tfr s. 
For structural reasons, we will 
focus on  the
 class of metaplectic \tfr s, because they  turn out to be a natural,
 almost  canonical  class of \tfr s and  are completely  characterized  by an
abstract covariance property~\cite{GS25}. 
The class of \mfr s  was recently  introduced  Cordero and
 Rodino~\cite{CorderoEtAl2022} as a  generalization  of the standard \tfr s and
became a powerful tool for the  investigation
of  \psdo s and function
spaces~\cite{CorderoRodino2022,CorderoRodino2023,CorderoGiacchi2024,CorderoEtAl2022,CorderoEtAl2024}
under the name Wigner analysis.

In this article, we study \up s for \mfr s. 
For the formulation of our
results, we start with \tfs s and the symplectic group.   For $(x,\omega)\in\Rd\times\Rd$, let
$\rho(x,\omega) = T_{x/2}M_\omega T_{x/2}  
$ denote the time-frequency shift acting on a function $f:\Rd\to\C$ by 
$$
\rho(x,\omega) f(t) = e^{-\pi i x\cdot \omega}e^{2\pi i \omega\cdot t}f(t-x).
$$
The symplectic group $\Spp\subseteq \R^{2d\times 2d}$ is the matrix group
\begin{equation}\label{eq:Spp_def}
\Spp = \left\lbrace \A\in \R^{2d\times 2d}: \A^t \J \A = \J \right\rbrace,\quad \J = \begin{psmallmatrix}
    0 & I \\ -I & 0
\end{psmallmatrix}.
\end{equation}
By the Stone-von Neumann theorem, for every $\A\in\Spp$,  there exists
a unitary operator $\hat{\A}$ on $\lrd $ satisfying the intertwining relation 
\begin{equation}
    \rho(\A\lambda) = \hat{\A} \, \rho(\lambda) \, \hat{\A} ^{-1},\quad \lambda \in\R^{2d}.
\end{equation} 
Thus $\hat{\A}$ describes how a \tfs\ is transformed under a linear coordinate change.

Now let $\bA \in \Sp$ be a $4d \times 4d $ symplectic
matrix\footnote{Note that we write $4d\times 4d$-matrices with bold letters.}, so that
$\bA $ acts on $\lrdd $. The \mfr s associated to $\bA $ is the
sequilinear form
\begin{equation}
  \label{eq:m1}
  W_{\bA } (f,g) = \hat{\bA } \,  (f\otimes g), \qquad \qquad f,g \in
  \lrd \, .
\end{equation}
The usual \tfr s, such as the \stft , the Wigner distribution, and
many others are special cases of \mfr s.

Our objective is the investigation of \up s for \mfr s $W_{\bA
}$. 
Qualitative \up s search for  decay or integrability  conditions on $W_{\bA
}(f,g)$ that force  $f\equiv 0$ or $g \equiv 0$.  A quantitative \up\
would be an inequality for $W_{\bA }(f,g)$ or force $f$ and $g$ to have a
certain shape. 
We will touch on both types in this paper.

In a first approach~\cite{GroechenigShafkulovska2024}, we studied the
Benedicks-Amrein-Berthier  \up\ for $W_{\bA }$   and developed a general method to study \up s for \mfr s. 
Our goal is to extend the scope of results and to cover several other types of \up s. Our emphasis is on the
method and not so much on the explicit 
results. 
Our main goal is to show that the methods developed in
\cite{GroechenigShafkulovska2024} are versatile enough to answer most questions about \up\ for metaplectic \tfr s. In the end, a reader should be convinced that other \up s can be derived in exactly the same way.

The  point of departure is the observation that the  validity of a particular \up\ depends on the
parametrizing matrix $\bA $.  Not all \mfr s 
fulfill a given type of \up .  The challenge then  is to characterize those \tfr s that obey a given \up . We provide  an  explanation which
\mfr s obey a given type of \up .

We recall that a  symplectic matrix is often
written in block form as $\bA =
\begin{psmallmatrix}
  A & B\\ C & D
\end{psmallmatrix}$ with $2d \times 2d$ blocks  and that every symplectic matrix  $\A \in \Sp $ possesses a
(non-unique) factorization into elementary matrices 
\begin{equation}
  \label{eq:nfactor}
  \bA =
  \begin{psmallmatrix}
I & 0 \\ Q & I    
  \end{psmallmatrix} \, \begin{psmallmatrix}
L & 0 \\ 0 & (L^t)^{-1}    
  \end{psmallmatrix} \, \begin{psmallmatrix}
A & B\\ -B & A    
  \end{psmallmatrix}
\end{equation}
with  $2d\times 2d $ blocks $Q,L,A,B$ such that  $Q=Q^t$, $L $ is invertible, and the
associated complex-valued matrix $U= A+iB$ is unitary.
This so-called pre-Iwazawa decomposition translates into a
factorization of the metaplectic operator $\hat \bA $ and a 
representation of the associated \mfr\ $\W _{\bA }$.

It turns out that
only the third factor $\begin{psmallmatrix}
A & B\\ -B & A    
  \end{psmallmatrix}$ is relevant for the validity of an \up .

 Our main insight about \up s for \mfr s is the following dichotomy depending solely on properties of $U^tU$.

  Let  $\bA \in \Sp$ with a factorization~\eqref{eq:nfactor} and let
  $U = A + iB \in U(2d,\C)$ be the associated unitary $2d\times 2d$
  matrix. 

  \textbf{Alternative  I.} If $U^tU$ is block-diagonal, i.e., there exist
  unitary $d\times d$-matrices $U_1, U_2 \in U(d, \C)$, such that $U^tU =
  \begin{psmallmatrix}
    U_1 &O\\ O & U_2
  \end{psmallmatrix}$, then there exist non-zero functions $f,g\in \lrd $, such that
  $W_{\bA } (f,g)$ has compact support and thus satisfies arbitrary
  decay or integrability conditions. Consequently, the usual types of
  \up s cannot hold for  such a \mfr .

  \textbf{Alternative II.}  If $U^tU$ is not block-diagonal, i.e.,  $U^tU =
  \begin{psmallmatrix}
    U_1 & V_1\\ V_2 & U_2
  \end{psmallmatrix}$ and either $V_1 \neq 0$ or $V_2 \neq 0$, then a meaningful and interesting \up\ can be formulated and proved  for $W_\A $.
      
  We will state explicitly the following types of \up s for a \mfr\  under the stringent \emph{hypothesis  that $U^tU$ is not block-diagonal}. 

\begin{enumerate}[(i)] 
    \item (Benedicks-type): If $\W_\bA(f,g)$ is supported on a set of
      finite measure, then $f\equiv 0$ or $g\equiv 0$. This was our
      main result in~\cite{GroechenigShafkulovska2024}.
    \item[(ii)] (Beurling-type): If $\W_\bA(f,g)$ satisfies
    \begin{equation}\label{eq:Beurling-intro1}
        \int_{\Rdd}|\W_\bA(f,g)(\lambda)|e^{\pi |\lambda\cdot M\lambda|} \, d\lambda <\infty
    \end{equation}
    for an appropriate matrix $\mmm=\mmm(\bA)$ determined by the 
    factorization \eqref{eq:nfactor} of $\bA $,  
    then $f\equiv 0$ or $g\equiv 0$, see Theorem~\ref{thm:beurling_WRu}.
    \item[(iii)] (Hardy-type): 
    Again, for an appropriate matrix $\Omega =\Omega (\bA)$ determined by the 
    factorization \eqref{eq:nfactor} of $\bA $, the condition
    $|\W_\bA (f,g)(\lambda)|\leq c e^{-\alpha \|\Omega \inv
      \lambda\|^2}$ for $\alpha >1$ implies that $f\equiv 0$ or
    $g\equiv 0$. 
    In the critical case $\alpha =1$, both $f$ and $g$ can
    be described explicitly, but their description is more intricate
    than in Hardy's original \up, see Corollary~\ref{cor:hardy_WRu} and Theorem \ref{thm:hardy_partial_quantitative}.
\end{enumerate}

In the second part, we study \up s for the quadratic  metaplectic \tf\
representations $\W_{\bA}(f,f)$ on $\Rdd$. Clearly, if $U^tU$ is not
block-diagonal, then the above results apply and yield $f\equiv
0$. However,  for quadratic \tfr s $\W _{\bA }$  a new
phenomenon emerges. Even if $U^tU=\begin{psmallmatrix}
    U_1 &O\\ O & U_2
  \end{psmallmatrix}$  is block-diagonal, some \up s for $W_{\bA
  }(f,f)$ do hold. A simple reduction leads  us to  investigate the equivalent
  question of  metaplectic pairs
  $(\hat{\A_1} f, \hat{\A_2} f)$ on $\Rd\times \Rd$ in place of a
  Fourier pair $(f, \hat{f})$. This case can be reduced to the
  assumption that $\A _1 = \mathrm{I}_{2d}$ and $\A _2 =
  \begin{psmallmatrix}
    V_1 & V_2 \\ -V_2 & V_1 
  \end{psmallmatrix}$ with $V_2 \neq 0$. Thus the unitary  matrix
  $V=V_1 +i V_2$ is not real-valued. A Beurling-type \up\ for the
  metaplectic pair $(f, \A _2 f)$ asserts that a condition of the form
  \begin{equation}\label{eq:Beurling-intro2}
  \int_{\Rdd} |(f\otimes \hat{\A }_2 f)(\lambda)|e^{ \pi| \lambda
    \cdot M\lambda| }\, d\lambda <\infty ,      
  \end{equation}
implies  $f\equiv 0$. 
Again, $M=M(\bA)$ is an appropriate matrix that comes from the
factorization \eqref{eq:nfactor}. As a corollary we obtain a
Hardy-type \up\ for the metaplectic  pair $(f,\hat{\A }_2 f)$
(Theorem~\ref{thm:hardy_pairs_1}). For a version in dimension $d=1$ we refer to the
recent work~\cite{gosson25}. A thorough analysis in higher
dimensions is contained in ~\cite{corderogm25}. It is  based on a
explicit formula of metaplectic operators
from~\cite{terMorscheOonincx2002O}. Despite similarities of the
results, it is not completely clear how they compare. 

As further examples, we formulate \up s according to {Gelfand-Shilov}
with super-exponential weights for the partial \stft\  
and a version of Nazarov's inequality for
metaplectic 
pairs, see Theorems~\ref{thm:gelfand_shilov_partial_STFT} and \ref{thm:nazarov}.

\textbf{Proof Method.} 
 We would like to emphasize that the above statements are just some
examples of \up s that can be proved. One could go on indefinitely,
but the proofs follow exactly the strategy outlined in
Section~3 and boil down to an exercise in meticulous,
but rather unpleasant book-keeping. In our perspective, the blueprint
for proving \up s in Section~\ref{rem:strategy} is our main contribution. 

The proof method was developed in \cite{GroechenigShafkulovska2024}
and is based on several different factorizations of a symplectic
matrix and the corresponding metaplectic operator. The most important
factorization is the so-called pre-Iwazawa decomposition
\eqref{eq:nfactor}.  An elaborate analysis of this factorization with
additional singular value decompositions yields an important
formula for a  \mfr\ $\W _{\bA }$. Given  to a symplectic matrix $\bA \in
\mathrm{Sp}(4d,\R )$, there exist matrices $\A , \B \in \Spp $ of smaller dimension,
so that $\hat{\A}, \hat{\B}$ act on $\lrd $, and $\Omega \in
\mathrm{GL}(2d,\R )$, such that
        \begin{equation}\label{eq:WRu-reductionb}
            |\W_{\bA}(f,g)| = |\hat{\bD}_{\Omega} V^k_{\hat\B g} \hat\A f|,\qquad f,g\in\ltd.
        \end{equation}
Here $V^k_gf$ is a \emph{partial} \stft , which  is a version of the
\stft\ over a lower-dimensional subspace, see Section~\ref{rem:strategy} for the
precise definition and its analysis.  
The line of argument then proceeds from (known) \up s for the \stft\ to prove analogous \up s for the partial \stft\ $V^k_gf$. The final \up\ for $\W _{\bA }$ then follows by a translation  by means of formula 
\eqref{eq:WRu-reductionb}. 
The technical details are hidden in the factorizations and the measure-theoretic subtleties of the partial \stft , but the proofs of the main \up s (Beurling, Hardy, Gelfand-Shilov) are then straightforward and rather short. 

For the treatment of quadratic \up s for $\W _{\bA } (f,f)$, we will analyze another factorization for unitary matrices $\begin{psmallmatrix}
A & B\\ -B & A    
  \end{psmallmatrix}$ when the block $B$ is invertible. 
  The corresponding metaplectic operators can be written as a tensor product of fractional Fourier transforms (see Lemma~\ref{lem:pair-to-partial} and its proof). In this context, we acknowledge our debt to P.\ Jaming~\cite{Jaming2022} who implicitly used such a factorization in dimension $d=1$ for pairs of fractional Fourier transforms.

\textbf{Outline.} Section~2 collects some facts about the symplectic
group and metaplectic operators. Section~3 presents several \up s for the partial \stft. 
In particular, Section~3.2 outlines the general proof strategy for \up s for \mfr s and is perhaps the most important message of this paper. In
Section~4 we prove \up s according to Hardy, Beurling, and
Gelfand-Shilov for general \mfr s. 
Section~5 focusses on quadratic \mfr s and contains \up s for metaplectic pairs of functions. Section~6 touches quantitative \up s in the style of Hardy and Nazarov. Further remarks in Section~7 and an appendix resuming a decomposition from ~\cite{GroechenigShafkulovska2024} complete the paper. 

\section{Preliminaries}
\subsection{Symplectic matrices and metaplectic operators} 
In this paper, we use the symmetric time-frequency shifts $\rho(x,\omega) = T_{x/2}M_\omega T_{x/2}$, $(x,\omega)\in\Rd\times\Rd$, which act on a function $f:\Rd\to\C$ by~\footnote{We use the symmetric \tf\ shifts for convenience sake; otherwise \eqref{eq:Stone_von-Neumann} would have additional phase factors.}
\begin{equation*}
\rho(x,\omega) f(t) = e^{-\pi i x\cdot \omega}e^{2\pi i \omega\cdot
  t}f(t-x)\, .  
\end{equation*}
The symplectic group $\Spp\subseteq \R^{2d\times 2d}$ is given by
\begin{equation}
\Spp = \left\lbrace \A\in \R^{2d\times 2d}: \A^t \J \A = \J \right\rbrace,\quad \J = \begin{psmallmatrix}
    0 & I \\ -I & 0
\end{psmallmatrix}.
\end{equation}
For every $\A\in\Spp$ there exists a unitary operator $\hat{\A}$ satisfying \cite[Chap.~4.2.]{Folland1989}
\begin{equation}\label{eq:Stone_von-Neumann} 
    \rho(\A\lambda) = \hat{\A} \rho(\lambda) \hat{\A} ^{-1},\quad \lambda \in\R^{2d}.
\end{equation} 
We call $\hat{\A}$ a metaplectic operator associated to $\A$.
Clearly $\A $ is not unique, but the Stone-von Neumann theorem asserts
that $\A $ is unique up to a multiplicative  factor of modulus
$1$. 

The set 
\begin{equation*}
\{ T: T \text{ is a metaplectic operator on } \ltd\}   
\end{equation*}
is a subgroup of the group of unitary operators acting on $\ltd$ and 
contains the metaplectic group $\Mpp$. This group of unitary
operators can be identified with  the twofold cover of $\Spp$, as
a subgroup, i.e.,  the projection
\begin{equation*}
    \pi^{\mathrm{Mp}}:\Mpp\to\Spp
\end{equation*}
is a surjective Lie group homomorphism  with kernel
$\mathrm{ker}(\pi^{\mathrm{Mp}}) =\left\lbrace \mathrm{id},
  -\mathrm{id} \right\rbrace$~\cite{Folland1989,Gosson2011}. Whereas for the construction of the metaplectic representation, the choice of the phase factor is essential, it usually does not matter in \tfa . 

The simplest types of symplectic matrices take the following form:
\begin{equation}\label{eq:Sp_standard_Ex}
\V_Q=\begin{psmallmatrix}
I & 0 \\
Q & I
\end{psmallmatrix},\quad
\D_L=\begin{psmallmatrix}
L & 0 \\
0 & L^{-T}
\end{psmallmatrix},\quad 
\Ro_{A+iB}= \begin{psmallmatrix}
    A & B \\ -B & A
\end{psmallmatrix},
\end{equation}
with $Q=Q^t\in\R^{d\times d}$,
$L\in\mathrm{GL}(d,\R)$~\footnote{$L^{-t} = (L^t)^{-1} = (L^{-1})^t$
  denotes the inverse of the transpose of $L$.}, and $U =
A+iB\in\mathrm{U}(d,\C)$~\footnote{Throughout, we split a unitary
  matrix into a sum of its real and its imaginary part.}. Note that the standard skew-symmetric matrix $\J$ is the orthogonal matrix $\Ro_{iI}$. 

Metaplectic operators can  have a rather complicated representation as an integral operator \cite{terMorscheOonincx2002O}. 
However, a significant simplification  is due to the fact that for any symplectic matrices $\A_1,\A_2\in\Spp$ and any particular choice of metaplectic operators $\hat{\A_1},\hat{\A_2}\in\Mpp$, there exists a phase  $c\in\T$ such that
\begin{equation}\label{eq:meta_decomp_std}
   \hat{\A_1 \A_2} = c\cdot   \hat{\A_1} \hat{\A_2}.
\end{equation}
Thus, to determine a metaplectic operator $\hat\A$ up to a phase factor, it suffices to factor a matrix $\A$ into elementary matrices $\A_1,\dots, \A_n$ with simple associated metaplectic operators $\hat{\A_1},\dots,\hat{\A_n}$.

The fundamental metaplectic operators are the Fourier transform $\F = \hat{\J}$, the linear "chirps" $\hat{\V_Q}$, and the linear coordinate changes $\hat{\D_L}$;
\begin{equation}\label{eq:standard_metaplectic}
    \begin{split}
        \hat{\D_L} f(x) &= \abs{\det{L}}^{-1/2} f(L^{-1}x),\qquad
    \hat{\V_Q} f(x)  = e^{\pi i x^t  Qx}f(x).
    \end{split}
  \end{equation}
Thus $\hat{\D_L}$ is just a linear change of coordinates and the
multiplication operator $\hat{\V _Q}$ satisfies $|\hat{\V_Q} f|  =
|f|$. When it comes to \up s, these operators are harmless, and we can
usually drop them. 

The matrices $\J$, $\V_Q$, $Q = Q^t\in\R^{d\times d}$,
and $\D_L$, $L\in\mathrm{GL}(d,\R)$, generate the symplectic group~\cite{Folland1989,Gosson2011},
thus any metaplectic operator can be factored, up to a phase, as a
(finite) product of $\F$, $\hat{\V_Q}$, and 
$\hat{\D_L}$ with  $Q = Q^t\in\R^{d\times d}$ and
$L\in\mathrm{GL}(d,\R)$. 
We emphasize that the phase factor in \eqref{eq:meta_decomp_std} is irrelevant because in the formulation of \up s only absolute values $|\hat\A f|$ arise. 
We will therefore omit $c$ whenever we apply \eqref{eq:meta_decomp_std}.

In addition to the factorization of a symplectic matrix  into the
elementary matrices  $\J,\V_Q,\D_L$, we need another decomposition of
$\Spp $, 
called the \emph{pre-Iwasawa decomposition} in~\cite{Gosson2011}.
  
\begin{proposition}[Pre-Iwasawa decomposition]\label{prop:pre-iwasawa}
      Let $\A\in\Spp$. Then there exist $Q=Q^t\in\R^{d\times d}$, $L\in\mathrm{GL}(d,\R)$ and $U\in\mathrm{U}(d,\C)$ such that
      \begin{equation}\label{eq:n3}
          \A=\V_Q\D_L\Ro_U.
      \end{equation}
  \end{proposition} 
  Since for an orthogonal matrix $V\in O(d,\R)$  we have the identity
  $$
  \D _V = \begin{psmallmatrix}
V & 0 \\
0 & V
\end{psmallmatrix}  = \Ro _{V+iO} \, ,
  $$
the pre-Iwazawa decomposition cannot be unique. 
To achieve uniqueness,
one must impose the additional condition $L=L^t>0$, in which case there
are explicit formulas for the constituent matrices $Q,L,U $ in
\eqref{eq:n3}, see~\cite[Prop.~250.]{Gosson2011}.

We note that each  mapping
\begin{equation}\label{eq:index_iso}
    Q\mapsto \V_Q,\qquad L\mapsto \D_L,\qquad U\mapsto \Ro_U,
\end{equation}
is group a isomorphism on the space of symmetric matrices with
addition, and on $\mathrm{GL}(d,\R)$ and on $\mathrm{U}(d,\C)$ with
multiplication. 
In addition, we  observe that  \footnote{Note that $\mathrm{O}(d,\R)=\mathrm{U}(d,\C)\cap \mathrm{GL}(d,\R)$.}
\begin{equation}\label{eq:DVRoV}
    \Ro_V = \D_V,\quad \Ro_{WUV} = \D_W \Ro_U \D_V,\qquad W,V\in\mathrm{O}(d,\R), U\in\mathrm{U}(d,\C).
  \end{equation}
  
\subsection{Tensors} 
Recall that $\ltd$ can be identified with any of the (Hilbert) tensor
products 
$$
L^2(\R^{d_1})\otimes \dots\otimes L^2(\R^{d_n}),\qquad \sum\limits_{j=1}^n d_j =d.
$$
The elementary tensors are the functions
\begin{equation}
    \bigotimes\limits_{j=1}^n f_j (x) = \prod\limits_{j=1} f_j(x_j),\qquad x_j\in\R^{d_j},\, f_j\in L^2(\R^{d_j}),\ 1\leq j\leq n.
\end{equation}
Most of the time, $n=2$ or $n=d$.
Throughout, we will also use the tensor product $T\otimes S$ of operators defined
on elementary tensors by 
\begin{equation}\label{eq:tensor_action_def}
    (T\otimes S) (f\otimes g) = Tf\otimes Sg,\qquad f\in L^2(\Rk),\ g\in L^2(\R^{d-k}).
\end{equation}
Using \eqref{eq:DVRoV} and the tensor notation introduced above, we can embed $\mathrm{Mp}(2k,\R)\times \mathrm{Mp}(2d-2k,\R)$ in $\Mpp$ by embedding the generators:
\begin{equation}\label{eq:op_embed}
    \hat{\V}_{\diag(Q_1,Q_2)} = \hat{\V}_{Q_1}\otimes \hat{\V}_{Q_2},\ \ \hat{\D}_{\diag(L_1,L_2)} = \hat{\D}_{L_1}\otimes \hat{\D}_{L_2},\ \ \hat{\Ro}_{\diag(U_1,U_2)} = \hat{\Ro}_{U_1}\otimes \hat{\Ro}_{U_2}.
\end{equation}
Note that \eqref{eq:op_embed} holds for a \emph{particular} choice of the generators, for instance, as chosen in \eqref{eq:standard_metaplectic}. 
Otherwise, there is a multiplicative constant to be taken into account. 
As we are interested in $|\hat\A f|$, this will not represent a problem later on.

\section{Qualitative uncertainty principles for the partial \stft}\label{sec:3}
In this section, we introduce the \emph{partial \stft}. 
We discuss qualitative \up s for partial \stft s  and explain why
these are the correct stepping stone to \up s for  general  metaplectic
\tf\ representations. 

\subsection{The partial \stft }
Given a function $f\in\ltd$ and writing $x=(x_1,x_2)$ for $x_1\in\R^k$ and $x_2\in\R^{d-k}$ 
for some $1\leq k< d$, we denote with $f_{x_2}$ the restriction~\footnote{If $k=d$, then we interpret $f_{x_2} $ as $f$.}
\begin{equation}
    f_{x_2}:\R^k\to\C,\quad x_1\mapsto f(x_1,x_2).
\end{equation}
By Fubini's theorem, the function
$f_{x_2}$ is square-integrable for almost all $x_2\in \R^{d-k}$. This fact allows us to define the partial short-time Fourier transform in $k$ variables almost everywhere.
\begin{definition}[Partial \stft] \label{def:partial_stft}
    Let  $f,g\in\ltd$ and let $k\in\{1,\dots, d\}$. 
    The partial short-time Fourier transform in the first $k$ variables of a function $f$ with respect to $g$ is defined as 
    \begin{equation}\label{eq:def_partial_STFT}
        \begin{split}
            V^k_{g}f(x_1,x_2, \omega_1,\omega_2) & = \int_{\R^k} f(t,x_2) \overline{g(t-x_1, -\omega_2)} e^{-2\pi i t\cdot \omega_1 }\, dt \\
            & = \langle 
            f_{x_2},
            M_{\omega_1} T_{x_1} 
            g_{-\omega_2}
            \rangle 
            = V_{
            g_{-\omega_2}
            } 
            f_{x_2}
            (x_1,\omega_1),
        \end{split}
    \end{equation}
    for $(x_1,x_2,\omega_1,\omega_2) \in \R^{k}\times \R^{d-k} \times \R^{k}\times \R^{d-k}$.
  \end{definition}
  If $k=d$, then the definition returns the usual \stft\ $V_gf$ on
  $\rdd $.
  
Similar to the full  \stft, $V_g^k f$ is a special case of a metaplectic \tf\ representation \cite{GroechenigShafkulovska2024}. 
By Fubini's theorem, for almost all $x_2,\omega_2\in\R^{d-k}$, $V_{g_{-\omega_2}} f_{x_2}(x_1,\omega_1)$ is a \stft\ of two square-integrable functions in $L^2(\R^k)$, hence all known uncertainty principles for the \stft \ are applicable to these restrictions. 

We first state a standard result in measure theory. 
\begin{lemma}\label{lem:measure_theory}
    Let $f,g\in\ltd$, $1\leq k\leq d$. Then the following statements hold.
    \begin{enumerate}[(i)]
        \item If the restriction $f_{x_2}$ is the zero function for almost all $x_2\in\R^{d-k}$, then $f\equiv 0$.
        \item Assume that $k<d$. If $f_{x_2}\otimes g_{\omega_2}\equiv 0$ for almost all $(x_2,\omega_2)\in\R^{2(d-k)}$, then $f\equiv 0$ or $g\equiv 0$. 
        \item If $V_g^k f\equiv 0$ almost everywhere, then $f\equiv 0$ of $g\equiv 0$ almost everywhere. 
    \end{enumerate}
    
\end{lemma}
 A self-contained proof of (iii)  can be found in
 \cite[Lem.~3.2.]{GroechenigShafkulovska2024}, (i) follows from
 Fubini's theorem, and (ii) follows from (i) after permuting the order
 of the variables. 

\subsection{The proof strategy.} \label{rem:strategy}     The primary reason we are interested in the partial \stft\ is the following proof {approach} that we already exploited in \cite{GroechenigShafkulovska2024} to prove a Benedicks-type uncertainty principle for $\W_{\bA} (f,g)$.
\begin{enumerate} 
    \item Factorization of $\A$: Compute a pre-Iwasawa decomposition $\bA = \bV_Q \bD_L \bRo_U$.
    \item Representation  of $W_\A$:  If $U^tU$ is not $d\times d$ block-diagonal, then there exists a $1\leq k\leq d$ and matrices $\A,\B\in\Spp$, $\Omega\in\mathrm{GL}(2d,\R)$, such that \cite[Thm.~1.1., Thm.~3.4.]{GroechenigShafkulovska2024}
    \begin{equation}\label{eq:free-to-stft}
        |\W_{\bA}(f,g)| = |\hat{\bA}(f\otimes \overline{g})|
         = | \hat{\bD}_{\Omega} (V^k_{\hat{\B} g}\hat{\A} f)|.
    \end{equation}
By setting   $\tilde f = \hat{\A} f$ and $\tilde g = \hat{\B} g$, an
\up\ for $\W_\A (f,g)$ is equivalent to an \up\ for
$V^k_{\tilde{g}}\tilde{f}$. 
    \item  Uncertainty principle for the partial \stft:  Show that
      the restriction $V^k_{\tilde g}\tilde f(\cdot , x_2,\cdot , \omega_2) 
    = V_{ \tilde{g}_{-\omega_2} } 
            \tilde{f}_{x_2}$
    satisfies the \up\  for almost all $(x_2,\omega _2)$  and
    conclude that $V^k_{\tilde{g}}\tilde{f} =0$. 
   \item Uncertainty principle for $W_\A $: Conclude  $ f=
     \hat{\A}^{-1} \tilde f\equiv 0$ or  $ g= \hat{\B}^{-1} \tilde
     g\equiv 0$. 
\end{enumerate}
We consider the proof outline with Steps (i) --- (iv) our  main
contribution to understand \up s for \mfr s. 
In this strategy, an \up\ for the partial \stft\ is an intermediate
result  towards an \up\ for a \mfr. Technically, we apply a known
\up\ for the full \stft\ to the restrictions $ V_{ \tilde{g}_{-\omega_2} } 
        \tilde{f}_{x_2}$ and then conclude that $V^k_gf \equiv 0$ and apply Lemma~\ref{lem:measure_theory}. The ensuing \up\ for $W_\bA $ is then a matter of diligent book-keeping, however, in view of \eqref{eq:free-to-stft}, its final formulation may look very ugly. 
            
        In the following sections, we execute this plan for the \up s of Hardy, Beurling, and Gelfand-Shilov. 

\subsection{Beurling-type} 
Beurling's \up\ is considered the strongest \up\ because the classical \up s\ such as Hardy's, Benedicks's, and Gelfand-Shilov's \up \ are easy consequences. 
The following version of Beurling's \up\ for the \stft\ was derived by Demange \cite[Thm.~1.2.]{Demange2005}.
\begin{theorem}\label{thm:beurling_STFT}
    Let $f,g\in\ltd$ and $N\geq 0$. 
    \begin{enumerate}[(i)] 
        \item If the weighted integral satisfies 
    \begin{equation}\label{eq:beurling_STFT}
        \int_\Rd \int_\Rd \frac{|V_g
          f(x,\omega)|}{(1+\norm{x}+\norm{\omega})^N}e^{\pi |x\cdot
          \omega|}\, dx d\omega <\infty \, ,
    \end{equation}
    then either $V_gf \equiv 0$ and thus  $f\equiv 0$ or $g\equiv 0$,
    or both can be written as 
    \begin{equation}
        f(x) = p_1(x)e^{-Ax\cdot x - 2\pi i \omega_0\cdot x},\quad 
        g(x) = p_2(x)e^{-Ax\cdot x- 2\pi i \omega_0\cdot x},\
    \end{equation}
    with  polynomials $p_1, p_2$ whose degrees satisfy 
    $ deg(p_1) + deg(p_2) < N-d$, $\omega_0\in\Cd$, and with  a real
    positive-definite  matrix $A$. 
    \item Qualitative version: In particular, if $N\leq d$ in \eqref{eq:beurling_STFT}, then $f\equiv 0$ or $g\equiv 0$.
    \end{enumerate}
\end{theorem}

The partial \stft \ satisfies the same type of an \up. 
We state a qualitative version and refer to Sections \ref{sec:quantitative} and \ref{sec:generalizations} for further discussion.
\begin{theorem}\label{thm:beurling_partial_STFT}
     Let $1\leq k \leq d$, $0\leq N\leq k$ and $f,g\in\ltd$. If
    \begin{equation}\label{eq:beurling_STFT_partial_L1}
        \int_{\R^{2d}} \frac{|V^k_g f(x,\omega)|}{(1+\norm{x}+\norm{\omega})^N}e^{\pi |x\cdot \omega|}\, dx d\omega <\infty,
    \end{equation}
    then $f\equiv 0$ or $g\equiv 0$.
\end{theorem}
\begin{proof} 
If $k=d$, then this is  Theorem \ref{thm:beurling_STFT} (ii). 
Thus we may assume $k<d$. 
Let $(x_1,\omega_1)\in\R^{2k}$ and $(x_2,\omega_2)\in\R^{2(d-k)}$. 
Fubini's theorem and the estimate 
\begin{equation}\label{eq:n5}
    \begin{split}
        \frac{1}{(1+\norm{x_1}+\norm{\omega_1})^N}\, e^{\pi|x_1\cdot \omega_1|}
        & \leq \frac{ e^{\pi|x_2\cdot \omega_2|}
          (1+\norm{x_2}+\norm{\omega_2})^N}{(1+\norm{x}+\norm{\omega})^N}\,
          e^{\pi|x_1\cdot \omega_1 +x_2\cdot \omega_2|}\\
        & = \frac{ e^{\pi|x_2\cdot \omega_2|}
          (1+\norm{x_2}+\norm{\omega_2})^N}{(1+\norm{x}+\norm{\omega})^N}\,
          e^{\pi|x\cdot \omega|} 
    \end{split}
\end{equation}
show that \eqref{eq:beurling_STFT_partial_L1} implies
    \begin{equation}\label{eq:beurling_STFT_partial}
        \int_{\R^{2k}} \frac{|V^k_g f(x_1,x_2,\omega_1,\omega_2)|}{(1+\norm{x_1}+\norm{\omega_1})^N}e^{\pi |x_1\cdot \omega_1|}\, dx_1 d\omega_1 <\infty
        \quad \text{for a.a. }(x_2,\omega_2)\in\R^{2(d-k)}.
    \end{equation}
Since $V^k_g f(x_1,x_2,\omega_1,\omega_2 ) =
V_{g_{-\omega_2}}f_{x_2}$, \eqref{eq:beurling_STFT_partial}  implies
that the condition \eqref{eq:beurling_STFT} is satisfied for almost
all $(x_2,\omega _2) \in \R ^{2(d-k)}$. By Theorem
\ref{thm:beurling_STFT},  $V_{g_{-\omega_2}}f_{x_2} \equiv 0$ for all
$(x_2,\omega_2)\in \R^{2(d-k)}\setminus S$.  
Consequently $V^k_gf (x_1,x_2,\omega _1,\omega _2) =0$ almost everywhere. 
From  Lemma \ref{lem:measure_theory}(iii)  we obtain that $f\equiv 0$ or $g\equiv 0$, as claimed.
\end{proof}
Note that both conditions  \eqref{eq:beurling_STFT_partial_L1} and
\eqref{eq:beurling_STFT_partial} imply $f\equiv 0$ or $g\equiv
0$. Condition \eqref{eq:beurling_STFT_partial_L1} treats the variables
$(x_1,\omega _2)$ and $(x_2,\omega _2)$ equally, whereas
\eqref{eq:beurling_STFT_partial} is asymmetric, but is weaker. Similar
observations can be made for the other \up s, see Section  
\ref{sec:generalizations}.

\subsection{Hardy-type}
As is well known for the Fourier pair $(f,\hat{f})$,  Beurling's
\up\ implies Hardy's \up. 
Similarly, Theorem \ref{thm:beurling_partial_STFT} implies a Hardy-type \up\ for the partial \stft. 
\begin{theorem}\label{thm:hardy_partial_1}
Let $1\leq k \leq d$, $\alpha,\beta,c>0$ with $\alpha\beta>1$.
If $f,g\in\ltd$ satisfy 
    \begin{equation}\label{eq:hardy_STFT_partial1}
    |V^k_g f(x,\omega)| \leq c e^{-\pi  (\alpha\norm{x}^2 +
      \beta\norm{\omega}^2)/2}\quad \text{ for a.a. }(x,\omega)\in\R^{2d},
    \end{equation}
    then $f\equiv 0$ or $g\equiv 0$. 
\end{theorem}
\begin{proof}
This is a direct consequence of Theorem \ref{thm:beurling_partial_STFT}. Since
\begin{equation}
    |x\cdot \omega| \leq \frac{\sqrt{\alpha}}{2\sqrt{\beta}}\norm{x}^2+\frac{\sqrt{\beta}}{2\sqrt{\alpha}}\norm{\omega}^2,
\end{equation}
we can estimate the integral in \eqref{eq:beurling_STFT_partial} with
$N=0$  as follows:
\begin{equation}
        \int_\Rd \int_\Rd |V_g^k f(x,\omega)| e^{\pi |x\cdot \omega|}\, dx d\omega 
    \leq  c \int_\Rd \int_\Rd e^{-\pi (1-(\alpha\beta)^{-1/2}) (\alpha \norm{x}^2+ \beta \norm{y}^2)/2} \, dx d\omega.
\end{equation}
This integral is finite for $\alpha\beta>1$, so that  Theorem
\ref{thm:beurling_partial_STFT} applies and yields that either $f\equiv 0$
or $g \equiv 0$.
\end{proof} 
Theorem \ref{thm:hardy_partial_1} can also be derived directly from
Hardy's \up\ for the \stft\
\cite{GroechenigZimmermann2001,BonamiEtAl2003}  in analogous fashion to the proof of  Theorem \ref{thm:beurling_partial_STFT}

\subsection{Gelfand-Shilov type}
Next, we prove a Gelfand-Shilov-type \up. In this case, Gaussian
weights as in Hardy's \up\ are replaced by other exponential weights. The
\up\ of Gelfand-Shilov for the Fourier pair $(f,\hat{f})$ is important
in  the theory of test functions and distributions~\cite{gelfand-vilenkin}, and an analogous
theory of test functions could be established by using the (partial)
\stft. 

\begin{theorem}\label{thm:gelfand_shilov_partial_STFT}
Let $1<p<\infty$, and let $q$ denote its conjugate H\"older exponent
$\tfrac{1}{p} + \tfrac{1}{q}=1$. If $f,g\in\ltd$ satisfy
\begin{equation}
    \begin{split}
    \int_\Rd\int_\Rd |V^k_g f(x,\omega)|e^{\frac{\pi}{p} \alpha ^p\norm{x}_p^p} \, dx d\omega <\infty, \\
    \int_\Rd\int_\Rd |V^k_g f(x,\omega)|e^{\frac{\pi}{q} \beta ^q\norm{\omega}_q^q} \, dx d\omega<\infty,
    \end{split}
\end{equation}
for  $\alpha,\beta>0$ satisfying  $\alpha \beta \geq 1$, then $f\equiv 0$ or $g\equiv 0$.
\end{theorem}
\begin{proof}
Using Young's inequality (in the exponent) and the convexity of the
exponential function, we estimate
\begin{align*}
  e^{\pi |x\cdot \omega |} &\leq e^{\pi \alpha \beta |x\cdot \omega |}
    \leq e^{\pi \big( \frac{1}{p} \|\alpha
    x\|_p^p + \frac{1}{q} \|\beta  \omega \|_q^q\big)}\\
&\leq \tfrac{1}{p}  e^{\pi \|\alpha   x\|_p^p} + \tfrac{1}{q} e^{\pi \|\beta  \omega \|_q^q}\, .
\end{align*}
Consequently,
\begin{align*}
\lefteqn{  \int_\Rd\int_\Rd |V^k_g f(x,\omega)|e^{\pi |x\cdot \omega|} \, dx
  d\omega}\\
  &\leq 
\tfrac{1}{p} \int_\Rd\int_\Rd |V^k_g f(x,\omega)|e^{\pi  \norm{\alpha x}_p^p} \, dx d\omega  
         + \tfrac{1}{q}\int_\Rd\int_\Rd |V^k_g
        f(x,\omega)|e^{\pi \norm{ \beta \omega}_q^q} \, dx
        d\omega <\infty \, .       
\end{align*}
Thus Theorem \ref{thm:beurling_partial_STFT} applies and yields
 $f\equiv 0$ or $g\equiv 0$.
\end{proof}

\section{Uncertainty principles for metaplectic \tf\ representations}\label{sec:joint}
Following the strategy outlined in Section~\ref{rem:strategy}, we now prove
several 
\up s for the \mfr s $W_\A $. The main tool is a factorization of
$W_{\A}$ and a subsequent reduction of the \up\ to the partial \stft.

    The following representation of $W_{\A} $ was proved in \cite[Lem.~3.3., Thm.~5.4.]{GroechenigShafkulovska2024}.
    \begin{lemma}\label{lem:WbA-to-partial-trivial}
          Let $\bA = \bV_Q\bD_L \bRo_U$ be a pre-Iwasawa decomposition
          of $\bA\in\Sp$.

          (i) Then
          \begin{equation} \label{eq:f1}
              |\W_\bA(f,g)(\lambda) |= |\hat{\bV}_Q \hat{\bD}_L \hat{\bRo}_U (f\otimes\overline{g})(\lambda)| = |\det L|^{-1/2}\, |\hat{\bRo}_U (f\otimes\overline{g})(L^{-1}\lambda)|.
          \end{equation}
       (ii)    Furthermore, if $U^tU \in\mathrm{U}(2d,\C)$ is $d\times d$ block-diagonal, then there exist matrices $W\in\mathrm{O}(2d,\R)$ and $V_1,V_2\in\mathrm{U}(d,\C)$ such that $U = W\cdot \diag(V_1,V_2)$, thus
          \begin{equation} \label{eq:n7}
        |\W_\bA(f,g) (\lambda)|
        = |\det L|^{-1/2} \, 
        \big|\big(\hat{\Ro}_{V_1} f\otimes \hat{\Ro}_{\overline{V_2}} g\big)(W^tL^{-1} \lambda)\big|.
          \end{equation}
    \end{lemma} 

    Lemma \ref{lem:WbA-to-partial-trivial} simplifies several
    aspects.

    Firstly, since 
    $|\widehat{\bV} _Q F| = |F|$ and $\hat{\bD}_L F(x) = |\det
    L|^{-1/2} F(L\inv x)$, these two operators do not
     change decay or integrability properties. 
   After dropping them in the factorization
   \eqref{eq:f1}, we may focus on 
   $\bRo_U$ for some $U\in\mathrm{U}(2d,\C)$.

    Secondly, there is a fundamental dichotomy for \up s for \mfr s
    depending solely on the nature of  $U^tU$. 

      Let us first settle the uninteresting case.

      \begin{lemma}
Let $\bA = \bV_Q\bD_L \bRo_U$ and assume that  $U^tU$ 
      is block-diagonal.     Then there exist non-zero functions $f,g\in\ltd$,  such that $\W_\bA(f,g)$ has compact support.    
      \end{lemma}
      \begin{proof}
    Just take $f_0,g_0\in C_c(\Rd)$ and set $f= \hat{\Ro}_{V_1}^{-1} f_0$ and $g = \hat{\Ro}_{\overline{V_2}}^{-1} g_0$. 
    Then by \eqref{eq:n7}, $|\W_\bA(f,g)| = |\hat\bD_{LW} (f_0\otimes g_0)|$ is a linear
    coordinate change of $f_0\otimes g_0$ and has compact support.     
      \end{proof}
    
    Thus, in this case we cannot expect an \up\ expressed by a decay
    or integrability condition to hold. For the treatment of the
    alternative, we use the following representation of $\W _{\bA } $.  

    \begin{lemma}\label{lem:WRu-to-partial} 
        Let $\bA=\bV_Q\bD_L\bRo_U\in\Sp$ with $U^tU\in \mathrm{U}(2d,\C)$ not $d\times d$ block-diagonal. 
        Then there exist an index $1\leq k\leq d$, symplectic matrices $\A,\B\in\Spp$ and a matrix $\Omega\in\mathrm{GL}(2d,\R)$ such that
        \begin{equation}\label{eq:WRu-reduction}
            |\W_{\bA}(f,g)| = |\hat{\bD}_{\Omega} V^k_{\hat\B g} \hat\A f|,\qquad f,g\in\ltd.
        \end{equation}
        The matrices $\Omega,\A,\B$ and the index $k$ can be derived explicitly from the pre-Iwasawa decomposition of $\bA$.
    \end{lemma}
    Lemma \ref{lem:WRu-to-partial} is a summary of a series of
    computations in \cite{GroechenigShafkulovska2024}. For
    completeness we explain  the exact derivation of the matrices  in the appendix.

Most \up s are related to decay or integrability conditions, which are
quite robust under linear changes of coordinates, and allow us to derive \up s for $\W_{\bA}(f,g)$ equivalent to \up s for the corresponding \stft. 
To illustrate this procedure, we derive a Beurling-type and a Hardy-type \up\ for $\W_{\bA}(f,g)$. 

Before we state a simple version of a Beurling-type \up\ for  $\W_\bA(f,g)$, we make a small observation. 
As we are going to discuss \up s for \emph{joint} \tf\
representations, it makes sense to express the conditions in terms of
$\lambda = (x,\omega)\in\Rdd$ instead of $x,\omega\in\Rd$. The inner
product on $\rdd $ is then written as 
    \begin{equation}\label{eq:beurling_joint}
        x\cdot \omega = \tfrac{1}{2}\lambda\cdot \begin{psmallmatrix}
            0 & I_d \\ I_d & 0
        \end{psmallmatrix} \lambda.
    \end{equation}
    In the following, we will use the right-hand side more often since it is better compatible with the change of coordinates of $\bD_\Omega$.
 
\begin{theorem}[Beurling-type]\label{thm:beurling_WRu}
    Let $\bA = \bRo_{U}\in \Sp$ and $\Omega=\Omega(\bA)$ be the matrix
    introduced in  the factorization \eqref{eq:WRu-reduction}.  Assume that $U^tU\in\mathrm{U}(2d,\C)$ is not $d\times d$ block-diagonal. If $f,g\in\ltd$ satisfy 
    \begin{equation}\label{eq:beurling_WRu}
        \int_{\Rdd} |\W_\bA(f,g)(\lambda)|e^{\pi|\Omega^{-1}\lambda\cdot  \begin{psmallmatrix}
            0 & I_d \\ I_d & 0
        \end{psmallmatrix}\Omega^{-1} \lambda|/2}\, d\lambda <\infty,
    \end{equation}
    then $f\equiv 0$ or $g\equiv 0$.
\end{theorem}
\begin{proof}
By Lemma \ref{lem:WRu-to-partial},
\begin{equation}
    |\W_\bA(f,g)| = 
    |\det \Omega|^{-1/2} |V^k_{\hat{\B}g} \hat{\A} f (\Omega^{-1}\cdot)|.
\end{equation}
Therefore, we can conclude from \eqref{eq:beurling_WRu}
 that
    \begin{equation}
       \begin{split}
       & |\det \Omega|^{1/2} 
         \int\limits_{\Rdd}|V^k_{\hat{\B} g} \hat{\A} f
         (\lambda)| \, e^{\pi|\lambda \cdot \begin{psmallmatrix}
            0 & I_d \\ I_d & 0
        \end{psmallmatrix} \lambda|/2}\, d\lambda \\
        \leq & |\det \Omega|^{-1/2} \int\limits_{\Rdd} |V^k_{\hat{\B} g} \hat{\A} f (\Omega^{-1}\lambda)|\, e^{\pi|\Omega^{-1}\lambda\cdot \begin{psmallmatrix}
            0 & I_d \\ I_d & 0
        \end{psmallmatrix} \Omega^{-1} \lambda|/2}\, d\lambda  <\infty.
       \end{split}
    \end{equation}
    By Theorem \ref{thm:beurling_partial_STFT}, $\hat{\A}f\equiv 0$ or
    $\hat{\B} g\equiv 0$. Since $\hat{\A}$ and $\hat{\B}$ are unitary, this is equivalent to $f\equiv 0$ or $g\equiv 0$, as desired.
\end{proof}

A Hardy-type \up \ follows directly from Theorem \ref{thm:beurling_WRu}. The proof is analogous to the proof of Theorem \ref{thm:hardy_partial_1}.
\begin{corollary}[Hardy-type]\label{cor:hardy_WRu}
      Let $\bA = \bV_Q\bD_L \bRo_{U}\in\Sp$ and $\Omega=\Omega(\bA)$ be the
      matrix in  the factorization \eqref{eq:WRu-reduction}.  Assume that $U^tU\in\mathrm{U}(2d,\C)$ is not $d\times d$ block-diagonal. If $f,g\in\ltd$ satisfy 
       \begin{equation}\label{eq:hardy_WRu_1}
         |\W_\bA(f,g)(\lambda)| \leq c e^{-\pi \alpha\norm{\Omega^{-1} \lambda}^2/2}\quad \text{for a.a. } \lambda\in\R^{2d},
    \end{equation}
    for some $c>0$ and $\alpha>1$,    
    then $f\equiv 0$ or $g\equiv 0$.
      \end{corollary}
   Alternatively, the Hardy-type \up\ follows from
   Theorem~\ref{thm:hardy_partial_1}. The case $\alpha =1$ in
   \eqref{eq:hardy_WRu_1} will be discussed in Section
   \ref{sec:quantitative}. 

\section{Qualitative uncertainty principles for quadratic \tf\ representations and pairs of metaplectic operators}\label{sec:pairs}

We next consider \up s for quadratic \mfr s $W_{\A} (f,f)$. 
 In this case, the situation is more complicated.

If $\bA = \bV_Q\bD_L\bRo_U$ and $U^tU$ is not block-diagonal, then \up
s for the \mfr\ $\W_\bA(f,f)$ are already covered by results in
Section \ref{sec:joint}. 

However, if $U^tU$ is block-diagonal, 
the situation is drastically different, as we have already noticed for
the Benedicks-type \up\ in~\cite{GroechenigShafkulovska2024}. The new phenomenon 
 is best illustrated by the Rihazcek distribution 
\begin{equation}
    R(f,g)(x,\omega) = e^{-2\pi i x\cdot
      \omega}f(x)\overline{\hat{g}(\omega)},\qquad x,\omega\in\Rd \, .
\end{equation}
and Benedicks's \up. 
The Rihazcek distribution is a metaplectic \tf\ representation with
block-diagonal $U^tU$~\cite{CorderoEtAl2022}. 
Clearly, $R(f,g)$ is  supported on a compact set for every  non-zero function $f$ and $\hat{g} \in\ltd$ with compact
support.  
However, the \up\ of Amrein-Berthier and Benedicks for the Fourier
pair $(f,\hat{f})$ 
states that $R(f,f)$ is supported on a set of finite measure if and
only if $f\equiv 0$~\cite{AmreinBerthier1977,Benedicks1985}.   

For the more general \mfr s $\W_\bA(f,f)$  it is a non-trivial
question when an  \up\ holds. 
In the following, we only consider the new case when  $\bA=\bRo_U$ and $U^tU$ is $d\times d$ block-diagonal. 
\begin{remark}[Reduction to pairs of operators]\label{rem:quadratic-to-pair} 
Let $\bRo_U\in \Sp$ with $U^tU\in\mathrm{U}(2d,\C)$. 
Using  Lemma \ref{lem:WbA-to-partial-trivial} (ii) and omitting the coordinate change $\bD _{WL}$, we may  assume without loss of
generality that $U = \diag(V_1,V_2)$ for some
$V_1,V_2\in\mathrm{U}(d,\C)$, and  we  study 
\begin{equation}
    |\hat\bRo_{\diag(V_1,V_2)} (f\otimes \overline{f})|
    = |\hat \Ro_{V_1} f| \otimes |\hat \Ro_{\overline{V_2}} f|.
\end{equation}

By setting $g = \hat \Ro_{V_1} f$ and $V = \overline{V_2}V_1^{-1} = \overline{V_2}V_1^{*}$, we rewrite this as
\begin{equation} \label{eq:n4}
    |\hat\bRo_{\diag(V_1,V_2)} (f\otimes \overline{f})|
    = |g| \otimes |\hat \Ro_{\overline{V_2}} \hat \Ro_{V_1}^{-1} \hat \Ro_{V_1} f |
    = |g| \otimes |\hat \Ro_{V} g|.
\end{equation}
After this reduction, it suffices to study  \up s for the pair $(g, \hat\Ro_V g)$. 
 \end{remark} 
 If $V$ is real-valued, i.e., $V\in\mathrm{O}(d,\R)$, then $
 \hat\Ro_V g(x) = g(V\inv x)$. Thus, if $g$ has compact support,
 then so does  $\hat\D_V g$, and no decay or integrability condition
 can force $g\equiv 0$. We will see that the real-valuedness of $V$ is the only obstruction
 to \up s for $W_{\bA } (f,f)$. For a characterization of those
 metaplectic matrices where the reduction in \eqref{eq:n4} yields a
 real-valued $V$ we refer to  \cite[Lem.~6.1.]{GroechenigShafkulovska2024}. 
\begin{remark}[Proof strategy --- quadratic version]\label{rem:strategy-quadratic} 
The strategy to prove \up s for quadratic \tfr s is similar to the
one presented in Section \ref{rem:strategy}, but  in the second step,
in \eqref{eq:free-to-stft}, we replace  the partial \stft\ with a
\emph{partial Fourier transform} which we recall next:

Let  $1\leq k<d$, the partial \ft \ $\F _k$  is defined as 
\begin{equation}\label{eq:partial_ft_action}
    \begin{split}
        \F_k f (\omega) 
     = \int_\Rk f(x_1,\omega_2) e^{-2\pi i \omega_1\cdot x_1}\,
    dx_1 
    = \F(f_{\omega_2})(\omega _1)
    = \hat{f_{\omega_2}}(\omega _1) ,\\
    \qquad\qquad \qquad\qquad\qquad\qquad\qquad
    \omega=(\omega_1,\omega _2)\in\Rk\times \R^{d-k},\ f\in L^1(\rd ).
    \end{split}
\end{equation}
This transform is a  metaplectic operator, namely, 
$$
\F_k =  \F\otimes \mathrm{id}_{L^2(\R^{d-k})}=\hat\Ro_{\diag(iI_k,
  I_{d-k})}\in\Mpp \, ,
$$
and therefore the partial \ft\  $\F _k$  extends
to a unitary operator on $\ltd $.  Note
that the restriction $f_{\omega _2}$ has $k$ free variables, therefore
we can interpret 
$\F _k$ as the  usual Fourier transform and  write $\F _k
f(\omega _1, \omega _2) = \widehat{f_{\omega _2}}(\omega _1)$ without ambiguity. 
If $k=d$, then we write $\F_d=\F$.
\end{remark}
\subsection{Uncertainty principles for the partial Fourier transform}
We now present a Beurling-type and a Hardy-type \up.  

This time, we begin with Hardy's \up, as its proof is 
straightforward. 

\begin{theorem}\label{thm:hardy_partial_pairs}
    Let $1\leq k\leq d$. If $f\in\ltd$ satisfies
\begin{equation}\label{eq:hardy_partial_pairs_1} 
        |f(x)|\leq c e^{-\pi \alpha \norm{x}^2},\quad 
        |\F_k f(\omega)|\leq  ce^{-\pi \beta \norm{\omega}^2 },\qquad \text{for a.a. }x,\omega\in\Rd 
    \end{equation}    
    for some constants $c>0$ and $\alpha\beta>1$, then $f\equiv 0$.
\end{theorem}
\begin{proof}[Proof of Theorem \ref{thm:hardy_partial_pairs}]
If $k=d$, this is Hardy's \up\ \cite{Hardy1933}, so  we assume $1\leq k<d$.
The assumption, combined  with Fubini's and Plancherel's theorem,
implies the following:  
there exist sets $S_x,S_\omega\subseteq\R^{d-k}$ of measure zero such
that for fixed  $x_2\in\R^{d-k}\setminus S_x, \omega_2 \in\R^{d-k}\setminus S_\omega$, 
the restrictions $f_{x_2}, \widehat{f_{\omega_2}} = (\F_k
f)_{\omega_2}$ lie in $L^2(\Rk)$,   and 
\begin{equation}\label{eq:hardy_partial_pairs}
\begin{split}
        |f(x_1,x_2)|\leq  c e^{-\pi \alpha \norm{x_1}^2},\quad \text{for a.a. } x_1\in\Rk,
        \\ 
        |\F_k f(\omega_1,\omega_2)|\leq  ce^{-\pi \beta
          \norm{\omega_1}^2},\quad  \text{for all  } \omega_1\in\Rk.
\end{split}
    \end{equation}
Then for $x_2\in \R^{d-k}\setminus (S_x\cup S_\omega)$ we have 
\begin{equation}\label{eq:hardy_partial_pairs_Step}
     \begin{split}
            |f_{x_2}(x_1)|&= |f(x_1,x_2)|\leq c e^{-\pi \alpha \norm{x_1}^2},\quad \text{for a.a. } x_1\in\Rk,\\ 
        |\widehat{(f_{x_2})}(\omega_1)| & = |\F_k f(\omega_1,x_2)|\leq  ce^{-\pi \beta \norm{\omega_1}^2} \quad  \text{for a.a. } \omega_1\in\Rk.
     \end{split}
    \end{equation}
    Hardy's \up\ for the Fourier pair $(f_{x_2}, \hat{f}_{x_2})$ now
    implies  $f_{x_2}\equiv 0$  for  $x_2\in \R^{d-k}\setminus (S_x\cup S_\omega)$.
Since $ S_x\cup S_\omega$ is a set of measure zero, Lemma
\ref{lem:measure_theory} (i) implies that $f\equiv 0$, as claimed. 
\end{proof}
The case $\alpha \beta =1$ will be treated in Section~\ref{sec:quantitative}.

Next, we formulate a Beurling-type theorem for the partial \ft . Its
proof is trickier. 
\begin{theorem}\label{thm:beurling_partial_pairs} 
     Let $1\leq k \leq d$ and $f\in\ltd$. If there exists an $0\leq N \leq k$ such that
\begin{equation}\label{eq:beurling_FT_partial_L1}
        \int_{\R^{2d}} \frac{|f (x) \F_k f(\omega)|}{(1+\norm{x}+\norm{\omega})^N}e^{2\pi |x\cdot \omega|}\, dx d\omega <\infty,
    \end{equation}
    then $f\equiv 0$ or $g\equiv 0$.
\end{theorem}
\begin{proof}
    Since $k=d$ is Beurling's \up\ \cite{BeurlingCollected1989,hormander91}, we assume that $1\leq k<d$. 
    We start as in
    Theorem~\ref{thm:beurling_partial_STFT}. Estimate~\eqref{eq:n5}
    implies that
    \begin{align}
\lefteqn{ \int_{\R^{2k}} \frac{|f (x_1,x_2) \F_k
     f(\omega_1,\omega_2)|}{(1+\norm{x_1}+\norm{\omega_1})^N}e^{2\pi
      |x_1\cdot \omega_1|}\, dx_1 d\omega_1} \\
      &\leq e^{2\pi |x_2\cdot
     \omega _2|} (1+\|x_2\|+\|\omega _2\|)^N \, 
    \int_{\R^{2k}} \frac{|f (x_1,x_2) \F_k
      f(\omega_1,\omega_2)|}{(1+\norm{x}+\norm{\omega})^N}e^{2\pi
      |(x_1,x_2)\cdot (\omega_1,\omega_2)|}\, dx_1 d\omega_1 \, ,       \label{eq:beurling_FT_partial}
    \end{align}
  and the last integral is finite for almost all $(x_2,\omega _2) \in
  \bR ^{2(d-k)}$ by     Fubini's theorem and  condition
  \eqref{eq:beurling_FT_partial_L1}. Call the exceptional set $S$. 
Whereas in \eqref{eq:hardy_partial_pairs_Step} we could use $x_2=\omega _2$, these variables are
now independent. To settle this difficulty, we resort to a Beurling-type \up\ for pairs of functions. It was shown in \cite[Cor.~3.7.]{Demange2005} that
    \begin{equation}
        \int_{\R^{2d}} \frac{{|g(x)h(\omega)|+|\hat g(x)\hat h(\omega)|}}{(1+\norm{x}+\norm{\omega})^N}e^{2\pi |x\cdot \omega|}\, dxd\omega<\infty \quad \implies \quad g\equiv 0 \text{ or } h\equiv 0.
    \end{equation} 
    We apply this result to the functions $g = f_{x_2}$ and $h = (\F_k f)_{\omega_2}=
    \widehat{f_{\omega _2}}$. Let 
    \begin{equation}
        \widetilde S = S\cup \{(\omega_2,x_2) : (x_2,\omega_2)\in  S \}\subseteq\R^{2(d-k)}.
    \end{equation} 
    This is also a set of measure zero. 
Then we obtain for all $(x_2,\omega_2)\in \R^{2(d-k)}\setminus \widetilde S$ 
    \begin{equation}
        \begin{split}
            & \phantom{=} \int_{\R^{2d}} \frac{{|f_{x_2}(x_1)(\F_k f)_{\omega_2}(\omega_1)|+|\widehat {f_{x_2}}(x_1)\widehat {(\F_k f)_{\omega_2}}(\omega_1)|}}{(1+\norm{x_1}+\norm{\omega_1})^N}e^{2\pi |x_1\cdot \omega_1|}\, dx_1d\omega_1 \\
            & = \int_{\R^{2d}} \frac{{|f_{x_2}(x_1)(\hat{f}_{\omega
                  _2})(\omega_1)|+|(\hat{
                  f}_{x_2}(x_1)f_{\omega_2}(-\omega_1)|}}{(1+\norm{x_1}+\norm{\omega_1})^N}e^{2\pi
              |x_1\cdot \omega_1|}\, dx_1d\omega_1 <\infty \, .
        \end{split}
    \end{equation}
    By\cite[Cor.~3.7.]{Demange2005}, $f_{x_2} \equiv 0$ or $(\F_k
    f)_{\omega_2}\equiv 0$. 
    Equivalently, $f_{x_2}\otimes (\F_k f)_{\omega_2}$ for all $(x_2,\omega_2)\in \R^{2(d-k)}\setminus \widetilde S$.
    Since $\widetilde S$ has measure zero, Lemma \ref{lem:measure_theory} (ii) implies $f\otimes \F_k f\equiv 0$. 
    As $\F_k$ is an invertible operator, this implies that $f\equiv 0$, as desired.
    \end{proof}

\subsection{Uncertainty principles for pairs of metaplectic operators and quadratic metaplectic representations}
In this subsection, we discuss uncertainty principles for $(f, \hat{\Ro}_V f)$, $V\in\mathrm{U}(d,\C)$ if $V\notin \mathrm{O}(d,\R)$. 

\begin{lemma} \label{lem:pair-to-partial}
      Let $\Ro_V\in\Spp$ and assume that $V\in U(d,\C )$ is not real-valued. Then there exist matrices 
      $\Omega
      \in\mathrm{GL}(2d,\R)$ and $\B
      \in\Spp$, such that 
      \begin{equation}
          |f|\otimes |\hat\Ro_V f|
          = \hat\D_{\Omega}\left(|\hat\B f| \otimes | \F_k \hat\B f|\right). 
      \end{equation}
      The matrices $\Omega$ and $\B$ 
      and the index $1\leq k\leq d$ can be determined explicitly.
\end{lemma}
\begin{proof}
By a version of a singular value decomposition
\cite[Thm.~5.1.]{FuerRzeszotnik2018}, there exist orthogonal matrices
$W_1,W_2\in\mathrm{O}(d,\R )$ and a diagonal unitary matrix $\Sigma =
\mathrm{diag}\, (\sigma _{1\, 1}, \dots , \sigma _{d\, d})
\in\mathrm{U}(d,\C)$ with $V = W_1 \Sigma W_2$.  
Since $V$ is not real-valued by assumption, neither is $\Sigma$, and
after a suitable permutation and multiplication with a diagonal
orthogonal matrix, we may  assume 
\begin{equation}\label{eq:sorted_imaginary_diag}
    \begin{split}
        \im \sigma_{1\,1}\geq \dots \geq \im \sigma_{k\, k} >0 &= \im \sigma_{k+1\ k+1} = \dots = \im \sigma_{d\,d},\ \text{and} \\
        1 & = \re \sigma_{k+1\ k+1} = \dots = \re \sigma_{d\,d},
    \end{split}
\end{equation}
where $k =\mathrm{rank}(\im(V))$.
By \eqref{eq:op_embed}, we can tensorize the action of $\hat\Ro_\Sigma$ as
    \begin{equation}\label{eq:decomp_fractional_multi}
        \hat\Ro_{\Sigma} = \bigotimes\limits_{n=1}^d \hat\Ro_{\sigma_{nn}}.
    \end{equation}
Note that in dimension $d=1$ for  $\sigma\in\T$,
$\hat\Ro_{\sigma}\in\mathrm{Mp}(2,\R)$ is the fractional Fourier
transform \cite{Namias1980,Wiener1929,Folland1989,OzaktasKutay2001}. 
If $\sigma \neq \pm 1$, we can decompose $\Ro_\sigma$ as
\begin{equation}
    \Ro_\sigma = \V_{\re(\sigma)/\im(\sigma)} \D_{\im(\sigma)} \J \V_{\re(\sigma)/\im(\sigma)},
\end{equation} 
see \cite{Jaming2022}, also
\cite[Sec.~4.1.]{GroechenigShafkulovska2024}.  Thus the fractional
\ft\ acts as  a product of  unimodular multipliers $\V
_{\re(\sigma)/\im(\sigma)}$ and the usual \ft.
If $\sigma=1$, then $\Ro_1 = I_2\in\mathrm{GL}(2,\R)$ is the identity matrix.
Using  \eqref{eq:decomp_fractional_multi}   and \eqref{eq:sorted_imaginary_diag}, we obtain the analogous multivariate factorization 
\begin{equation}
    \Ro_\Sigma = \V_C \D_B \Ro_{\diag(i I_k, I_{d-k})} \V_C,
\end{equation}
where $B,C\in\R^{d\times d}$ are the diagonal matrices 
\begin{equation}
    b_{nn} = \begin{cases}
        \im(\sigma_{nn}),&\quad 1\leq n\leq k,\\
        1, & \quad n>k,
    \end{cases},
    \quad 
      c_{nn} = \begin{cases}
        \re(\sigma_{nn})/\im(\sigma_{nn}),&\quad 1\leq n\leq k,\\
        0, & \quad n>k.
    \end{cases}
\end{equation}
Since $|\hat V_C f| = |f|$, we can write
\begin{equation}
    \begin{split}
        |f|\otimes |\hat\Ro_V f| 
        & = |f|\otimes |\hat\D_{W_1}\hat{\Ro }_\Sigma \hat\D_{W_2} f| \\
        & = |\hat{\D }_{W_2}^{-1}\hat{\D }_{W_2}f|\otimes |\hat\D_{W_1}\hat\V_C \hat\D_B\hat\Ro_{\diag(iI_k, I_{d-k})} \hat\V_C\hat\D_{W_2} f| \\
        & = \hat{\D}_{\diag(W_2^t, W_1)}
        \left(
        |\hat\V_C \hat \D_{W_2} f|\otimes |\hat\D_B\hat\Ro_{\diag(iI_k, I_{d-k})} \hat\V_C\hat\D_{W_2} f|
        \right).
    \end{split}
\end{equation}
Now set $\Omega = \diag(W_2^t, W_1 B)\in \mathrm{GL}(2d,\R)$ and $\B
= \V_C \D_{W_2}\in \mathrm{Sp}(2d,\R )$.
\end{proof}
Just as in the case of the sesquilinear forms, the combination of  the
factorization of  Lemma \ref{lem:pair-to-partial} with the \up s of
Theorems \ref{thm:beurling_partial_pairs} and
\ref{thm:hardy_partial_pairs} generates \up s for pairs of metaplectic
operators. 
Equivalently, by Remark \ref{rem:quadratic-to-pair}, these
can be formulated as \up s for quadratic metaplectic \tf \ representations.  

\begin{theorem}[Beurling-type]\label{thm:beurling_pairs} 
Let $\Ro_V\in\Spp$ and $\Omega = \Omega(V)$ be the matrix introduced
in Lemma \ref{lem:pair-to-partial}. Assume that $V\in U(d,\C)$ is not real-valued. If $f\in\ltd$ satisfies
\begin{equation}\label{eq:beurling_pair_1} 
     \int_{\Rdd} |(f\otimes \hat\Ro_V f)(\lambda)|e^{ \pi|\Omega^{-1} \lambda\cdot  \begin{psmallmatrix}
            0 & I_d \\ I_d & 0
        \end{psmallmatrix} \Omega^{-1} \lambda|}\, d\lambda <\infty ,
\end{equation}
then $f\equiv 0$. 
\end{theorem}
\begin{proof}
    The proof follows the same pattern as Theorem \ref{thm:beurling_WRu}. 
    In the notation of Lemma \ref{lem:pair-to-partial}, the hypothesis \eqref{eq:beurling_pair_1} is equivalent to 
    \begin{equation}\label{eq:beurling_pair_2}
     \int_{\Rdd} |(\hat\B f\otimes \F_k\hat \B f)(\lambda)|e^{\pi|\lambda\cdot \begin{psmallmatrix}
            0 & I_d \\ I_d & 0
        \end{psmallmatrix}\lambda|}\, d\lambda <\infty.
\end{equation}
Noting that the exponent equals $2\pi |x\cdot \omega |$, we  apply Theorem \ref{thm:beurling_partial_pairs} and obtain $\hat\B
f\equiv 0$, whence  $f\equiv 0$.
\end{proof}

\begin{theorem}[Hardy-type]\label{thm:hardy_pairs_1}
Let $\Ro_V\in\Spp$ and $\Omega = \Omega(V)$ be the matrix introduced in Lemma \ref{lem:pair-to-partial}. 
Assume that $V$ is not real-valued. If $f\in\ltd$ satisfies
\begin{equation}\label{eq:hardy_pair_1}
     |(f\otimes \hat\Ro_V f)(\lambda)| \leq c e^{ - \pi \alpha |\Omega^{-1} \lambda \cdot \diag(\alpha I_d,\beta I_d) \Omega^{-1} \lambda|}\quad \text{for a.a. }\lambda\in\Rdd
\end{equation}
for some $\alpha,\beta, c>0$ with $\alpha\beta>1$, 
then $f\equiv 0$. 
\end{theorem}
\begin{proof}
    In the notation of Lemma \ref{lem:pair-to-partial}, \eqref{eq:hardy_pair_1} is equivalent to 
    \begin{equation}
        |\hat \B f \otimes \F_k\hat\B f|(x,\omega) \leq \tilde c e^{- \pi (\alpha \norm{x}^2+\beta \norm{\omega}^2)} \quad \text{for a.a. }\lambda = (x,\omega)\in\Rdd.
    \end{equation}
    By Theorem \ref{thm:hardy_partial_pairs}, $\hat\B f\equiv 0$. Equivalently, $f\equiv 0$.
\end{proof}
Note that Theorem \ref{thm:hardy_pairs_1} also follows from Theorem \ref{thm:beurling_pairs}.

Theorems~\ref{thm:beurling_pairs} and \ref{thm:hardy_pairs_1}  are \up s for a symplectic pair $(f,
\hat\Ro_V)$ in place of a Fourier pair $(f, \hat{f})$. We leave it to
the reader to translate these results to \up\ for $\W _{\bA }$ by
reversing the steps of the reduction in
Remark~\ref{rem:quadratic-to-pair}. This amounts to an unpleasant and 
by no means  trivial work of book-keeping that, in our opinion,  does
not seem to offer any new insight. 

For a comparison to other approaches to Hardy's \up\ for metaplectic
pairs we refer to \cite{gosson25,corderogm25}. 

\section{Quantitative uncertainty principles}\label{sec:quantitative}
So far, we have discussed qualitative \up s, in other words,  a condition on
$\W_{\bA } (f,g)$ implies that either $f$ or $g$ vanishes. In some
cases, for instance, in a limiting case of Hardy's \up , we can
derive a certain shape for $f$ and $g$ and obtain a  quantitative \up
s. In this section we offer two examples of such quantitative \up
s. Its purpose is to demonstrate that the method outlined in Section~\ref{sec:3}
is also applicable to quantitative \up s. 

We first  present a complete version of a Hardy-type \up\ for the
partial \stft .

\begin{theorem}\label{thm:hardy_partial_quantitative}   
Let $1\leq k\leq d$ and assume that $f,g\in\ltd$ satisfy
     \begin{equation}
         |V_g^k f(\lambda)|\leq c (1+\norm{\lambda})^N e^{-\pi \alpha
           \norm{\lambda}^2/2} \qquad \text{ for a. a. }\lambda \in\R^{2d}
     \end{equation}
     for some $\alpha,c >0$ and $N\geq 0$.
     
     (i)     If $\alpha>1$, then $f\equiv 0$ or $g\equiv 0$.

     (ii) If $\alpha =1$, then one of the following occurs. 
     Either $f\equiv 0$ or $g\equiv 0$, or neither $f$ nor $g$ is the zero function and both can be written as
     \begin{equation}\label{eq:Hardy_partial_full}
         \begin{split}
             f(x_1, x_2)  & = p(x_1 ,x_2) e^{2\pi i \omega_0\cdot x_1} e^{-\pi (x_1-x_0)A (x_1-x_0)} \quad \text{for a. a. }(x_1,x_2)\in\Rd,\\
             g(x_1,x_2) & = q(x_1 ,x_2) e^{2\pi i \omega_0\cdot x_1} e^{-\pi (x_1-x_0)A (x_1-x_0)}
             \quad \text{for a. a. }(x_1,x_2)\in\Rd.
         \end{split}
     \end{equation}
with  polynomials $p(\cdot, x_2), q(\cdot,\omega_2)$  whose degrees satisfy 
$\mathrm{deg}(p_{x_2})+\mathrm{deg}(q_{-\omega_2})\leq N$ for almost all $(x_2,\omega_2)\in\R^{2(d-k)}$, $x_0, \omega_0\in\C^k$, and $A\in\R^{k\times k}$ a real positive definite matrix. 

(iii) If, in addition,  $k=d$ and $\alpha =1$, then $p$ and $q$ are polynomials,
i.e., $f$ and $g$ are finite linear combinations of Hermite functions.
 \end{theorem}
 \begin{proof} 
   (i)  is covered by Theorem  \ref{thm:hardy_partial_1}.

   (iii) is ~\cite[Cor.~6.6]{BonamiEtAl2003}.  We will use it in
   the proof of the new case (ii). 

   (ii) Let  $k<d$ and  $\alpha =1$. If either $f\equiv 0$ or
   $g\equiv 0$, there is nothing to prove. So assume that both
   functions are non-zero. This means that the sets
        \begin{equation}
         \begin{split}
             M_1 = \{ x_2\in \bR ^{d-k}: f_{x_2} \in L^2(\bR
             ^{d-k}), f_{x_2} \not \equiv 0\} \\
             M_2 = \{ \omega_2\in \bR ^{d-k}: g_{-\omega_2} \in L^2(\bR
             ^{d-k}), g_{-\omega_2} \not \equiv 0\} 
         \end{split}
     \end{equation}
     both have positive measure in $\bR ^{d-k}$. On the complement of
     $M_1$ the restriction $f_{x_2}$ vanishes, except on a set of
     measure $0$, and likewise for $M_2$.  If $(x_2,\omega_2)
     \in M_1\times M_2$, then the assumption implies that
          \begin{equation}\label{eq:Hardy_partial_implied1}
         \begin{split}
              & |V_{g_{-\omega_2}} f_{x_2}(x_1,\omega_1)| =
              |V_g^kf(x_1,x_2, \omega _1, \omega _2)|\\
              & \leq c (1+\norm{(x_1,x_2,\omega_1,\omega_2)})^N e^{-\pi\alpha(\norm{(x_1,\omega _1)}^2 +\norm{ (x_2,\omega_2)}^2)/2}\\
              & \leq c e^{-\pi\alpha \norm{(x_2,\omega_2)}^2}(1+\norm{(x_2,\omega_2)})^N\cdot (1+\norm{(x_1,\omega_1)})^N e^{-\pi\alpha \norm{x_1,\omega_1}^2/2}
         \end{split}
     \end{equation}
for all $(x_1,\omega_1)\in\R^{2k}$. By a theorem of Bonami, Jaming,
and Demange~\cite{BonamiEtAl2003} for the full \stft, there exist polynomials
$p_{x_2}$ and $q_{-\omega _2}$ with degrees
$\mathrm{deg}(p_{x_2})+\mathrm{deg}(q_{\omega_2})\leq N$ and a
generalized Gaussian function, such that
\begin{align*}
f_{x_2}(x_1) &= p_{x_2} e^{2\pi i \omega_0\cdot x_1} e^{-\pi
  (x_1-x_0)A (x_1-x_0)} \\
g_{-\omega_2}(x_1) & = q_{-\omega_2} e^{2\pi i \omega_0\cdot x_1} e^{-\pi
  (x_1-x_0)A (x_1-x_0)} \, .
  \end{align*}
Since  $f_{x_2} \equiv 0$ for $x_2\not\in M_1$ and $g_{-x_2} \equiv 0$
for $x_2 \not \in M_2$, the functions $f,g$ have the claimed form. 
 \end{proof}
 In addition, one can say that the functions $p$ and $q$ are
 polynomials in $x_1$ and possess Gaussian decay in $x_2$, but 
 nothing more can be said.
 Evidently, with minor adjustments, one can state a quantitative Beurling-type \up.

 As the second example, we discuss Nazarov's \up\ for metaplectic pairs. Nazarov's \up\ \cite{Nazarov1993} is a quantitative version of Benedicks's \up.
For a Fourier pair $(f, \hat{f})$ its multidimensional version from  \cite{Jaming2007} is encapsulated in the following inequality:  for all sets $S,T\subseteq\Rd$ of finite measure
\begin{equation}\label{eq:nazarov_multi_ft}
    \int_\Rd |f(x)|^2 dx 
    \leq C e^{C\min(|S|\cdot |T|,|S|^{1/d}w(T),|T|^{1/d}w(S))} 
    \left( \int_{\Rd\setminus S} |f(x)|^2 dx+ \int_{\Rd\setminus T} |\hat f(x)|^2 dx
    \right).
\end{equation}
Here $C$ is an absolute constant and $w(S)$ denotes the mean width of
a measurable set $S\subseteq\Rd$.  We refer to \cite{Jaming2007} for
further details. 
To simplify the statement, we denote with $\nc(S,T)$ the quantity 
$$
\nc(S,T) = C e^{C\min(|S|\cdot |T|,|S|^{1/d}w(T),|T|^{1/d}w(S))}.
$$ 
We now state a Nazarov-type \up\ for pairs of metaplectic operators. 
\begin{theorem}[Nazarov-type] \label{thm:nazarov}
Let $\A_1 = \V_{Q_1} \D_{L_1}\Ro_{U_1}$, $\A_2 = \V_{Q_2} \D_{L_2}\Ro_{U_2}\in \Spp$. 
If the imaginary part $\im(U)$ of $U = U_2U_1^*$ is invertible, then $(\hat\A_1 f, \hat\A_2 f)$ satisfies a Nazarov-type \up: 
\begin{align}\label{eq:nazarov_meta_pairs}
\lefteqn{     \int_{\Rd}|f(x)|^2 dx  }\\
    & \leq \nc(L_1^{-1} S, \im(U)^{-1} L_2^{-1} T)\left( \,\int_{\Rd\setminus S} |\hat \A_1 f(x)|^2 dx + \int_{\Rd\setminus T} |\hat\A_2 f(x)|^2 dx
    \right).
\end{align}
\end{theorem}
\begin{proof}
Let $f\in\ltd$ and set $g=\hat\Ro_{U_1} f$. By Lemma
\ref{lem:WbA-to-partial-trivial} and \eqref{eq:n4} 
\begin{equation}
    |\hat\A_1 f|\otimes |\hat \A_2 f| 
    =  |\hat \D_{L_1}\hat \Ro_{U_1} f|\otimes |\hat \D_{L_2} \hat\Ro_{U_2} f| 
    = |\hat \D_{L_1} g|\otimes |\hat \D_{L_2} \hat\Ro_{U} g|.
\end{equation}
Thus
\begin{equation}
    \begin{split}
    \mathcal{I} & \coloneqq \,\int_{\Rd\setminus S} |\hat \A_1 f(x)|^2 dx + \int_{\Rd\setminus T} |\hat\A_2 f(x)|^2 dx \\
    & = \,\int_{\Rd\setminus S} |\hat \D_{L_1} g(x)|^2 dx + \int_{\Rd\setminus T} |\hat \D_{L_2} \hat\Ro_{U} g(x)|^2 dx \\
    & = \,\int_{\Rd\setminus L_1^{-1} S} | g(x)|^2 dx + \int_{\Rd\setminus L_2^{-1} T} |\hat\Ro_{U} g(x)|^2 dx.
    \end{split}
\end{equation}
We apply Lemma \ref{lem:pair-to-partial} to $\Ro_U$ and use  the singular value  decomposition  $U = W_1 \Sigma W_2$, where
$W_1,W_2\in\mathrm{O}(d,\R)$ and  $\Sigma\in \mathrm{U}(d,\C ) $ is a diagonal matrix. 
Since $\im(U)$ is invertible by assumption, so is $\im(\Sigma)$. 
Recall that the matrices $\B$ and $\Omega $ in Lemma~\ref{lem:pair-to-partial} were given by  
$\B =\V_{\im(\Sigma)^{-1}\re(\Sigma)} \D_{W_2}\in\Spp$ and $\Omega = \diag(W_2^t, W_1 \im(\Sigma))$, which is block-diagonal. 
Thus Lemma~\ref{lem:pair-to-partial} says that  
      \begin{equation}
          |g|\otimes |\hat\Ro_U g| = \hat\D_{\diag(W_2^t, W_1 \im(\Sigma))}\left(|\hat\B g| \otimes | \F \hat\B g|\right). 
      \end{equation}
 We can now  further transform $\mathcal{I}$ as
\begin{equation}
    \begin{split}
        \mathcal{I}
        & = \,\int_{\Rd\setminus L_1^{-1} S} |\hat\D_{W_2^t} \hat \B g(x)|^2 dx + \int_{\Rd\setminus L_2^{-1} T} |\hat\D_{W_1} \hat\D_{\im(\Sigma)} \F \B g(x)|^2 dx \\
        & = \,\int_{\Rd\setminus W_2 L_1^{-1} S} |\hat \B g(x)|^2 dx +
        \int_{\Rd\setminus \im(\Sigma)^{-1} W_1^{t} L_2^{-1} T}
        |\F\hat \B g(x)|^2 dx. 
    \end{split}
\end{equation}
We now use $\|f\|_2 =  \|g\|_2= \|\hat{\B } f \|_2$ and  apply Nazarov's
inequality~\eqref{eq:nazarov_multi_ft} to $\hat\B g$ as follows: 
\begin{equation}
\begin{split}
        \int_{\Rd}|f(x)|^2 dx 
        &         =  \int_{\Rd} | \hat \B g(x)|^2 dx   \\
        & \leq \nc(W_2 L_1^{-1} S, \im(\Sigma)^{-1} W_1^{t} L_2^{-1} T)   \\
        &\qquad \left(\,
        \int_{\Rd\setminus W_2 L_1^{-1} S} |\hat \B g(x)|^2 dx + \int_{\Rd\setminus \im(\Sigma)^{-1} W_1^{t} L_2^{-1} T} |\F\hat \B g(x)|^2 dx
        \right) \\
        &  = \nc(W_2 L_1^{-1} S,\im(\Sigma)^{-1} W_1^{t} L_2^{-1}  T) \\
    &\quad \left( \,\int_{\Rd\setminus S} |\hat \A_1 f(x)|^2 dx + \int_{\Rd\setminus T} |\hat\A_2 f(x)|^2 dx
    \right),
\end{split}
\end{equation}
Since $\im(U) =W_1\im(\Sigma)W_2$ and the constant $\nc(\cdot,\cdot)$ is invariant under rotations, 
\begin{equation}
    \nc(W_2 L_1^{-1} S,\im(\Sigma)^{-1} W_1^{t} L_2^{-1} T ) = \nc(L_1^{-1} S, \im(U)^{-1} L_2^{-1} T).
\end{equation}
This concludes the proof.
\end{proof}

\section{Relaxations and generalizations}\label{sec:generalizations} 
\subsection{Some relaxation of conditions} 
In the proof of  Theorem \ref{thm:beurling_partial_STFT}   we saw that
the symmetric condition  \eqref{eq:beurling_STFT_partial_L1} could be
replaced by the weaker condition
\begin{equation}\label{eq:beurling_STFT_partial-relaxed}
        \int_{\R^{2k}} \frac{|V^k_g f(x_1,x_2,\omega_1,\omega_2)|}{(1+\norm{x_1}+\norm{\omega_1})^N}e^{\pi |x_1\cdot \omega_1|}\, dx_1 d\omega_1 <\infty
        \quad \text{for a.a. }(x_2,\omega_2)\in\R^{2(d-k)}.
      \end{equation}
for \emph{almost all cross-sections} of $V^k_gf$. Such a relaxation
also works for other \up s stated, e.g. for Theorem
\ref{thm:hardy_partial_1} and Theorem
\ref{thm:gelfand_shilov_partial_STFT}, for the quadratic \tfr s in
Section \ref{sec:pairs},  and  the  quantitative \up s in Theorem
\ref{thm:hardy_partial_quantitative}. 

We observe that condition~\eqref{eq:beurling_STFT_partial-relaxed}  is
the maximal relaxation. One can show that for all $\varepsilon>0$ there exist non-zero functions $f=g\in\ltd$ which violate \eqref{eq:beurling_STFT_partial-relaxed} on a measurable set $P\subseteq\R^{2(d-k)}$ of measure $\varepsilon$.
The easiest examples are functions of the type 
   $f(x_1,x_2) = g(x_1,x_2) = e^{-\pi x_1\cdot x_1}\psi(x_2)$, where $\psi\in C_c([-\varepsilon,\varepsilon]^{d-k})$ is a non-zero even function. The same prototype indicates the limitations of all other \up s discussed in this paper.

   \subsection{The choice of the matrix}
The exponential factor in \eqref{eq:beurling_STFT_partial-relaxed} can
be replaced by more complicated expressions. Note that 
\begin{equation}
    x_1\cdot \omega_1 = \lambda_1\cdot \tfrac{1}{2}\begin{psmallmatrix}
        0 & I_k \\ I_k & 0
    \end{psmallmatrix} \lambda_1 
    = \lambda \cdot \tfrac{1}{2}\begin{psmallmatrix}
        0 & 0 & I_k & 0 \\
        0 & 0 & 0 & 0 \\
        I_k & 0 & 0 & 0 \\
        0 & 0 & 0 & 0
    \end{psmallmatrix} \lambda,\quad \lambda_1=(x_1,\omega_1),\
    \lambda = (x_1,x_2,\omega_1,\omega_2). 
\end{equation}
Using  the symplectic invariance of the \stft \cite{Gosson2011}, 
$    |V_g f| = |\bD_{\A^{-1}} V_{\hat \A g} \hat\A f| $ for $ f,g\in
    L^2(\Rk)$ and $ \A\in\mathrm{Sp}(2k,\R)$,  partial \stft\ can be
    be written as   
\begin{equation}\label{eq:V_gf_Sp_cov}
\begin{split}
    &\phantom{=}\    |V^k_g f(x_1,x_2,\omega_1,\omega_2)| 
    = |V_{g_{-\omega_2}} f_{x_2}(x_1,\omega_1)|
    = |\bD_{\A^{-1}} V_{\hat \A g_{-\omega_2}} \hat\A f_{x_2}(x_1,\omega_1)| \\
    & = |\bD_{\A^{-1}\otimes \mathrm{id}} V^k_{(\hat \A\otimes \mathrm{id}) g} (\hat\A\otimes \mathrm{id}) f(x_1,x_2,\omega_1,\omega_2)|,
\end{split}
\end{equation}
where $\A^{-1}\otimes \mathrm{id}\in \mathrm{GL}(2d,\R)$ is the
corresponding change of coordinates. Consequently,   we can replace $\tfrac{1}{2}\begin{psmallmatrix}
    0 & I_k \\ I_k & 0
\end{psmallmatrix}$   with 
\begin{equation}
\lambda_1\cdot \A^{t}\tfrac{1}{2}\begin{psmallmatrix}
        0 & I_k \\ I_k & 0
    \end{psmallmatrix}\A \lambda_1 
    = \lambda \cdot \tfrac{1}{2}(\A\otimes\mathrm{id})^t \begin{psmallmatrix}
        0 & 0 & I_k & 0 \\
        0 & 0 & 0 & 0 \\
        I_k & 0 & 0 & 0 \\
        0 & 0 & 0 & 0
    \end{psmallmatrix} (\A\otimes\mathrm{id}) \lambda,
    \quad \A\in\mathrm{Sp}(2k,\R).
    \end{equation}
and obtain more complicated formulations of Beurling-type and
Hardy-type \up s. 
 
\section{Appendix}
We indicate a proof of  Lemma~\ref{lem:WRu-to-partial}. This lemma
asserts that the  pre-Iwasawa decomposition  $\bA = \bV_Q\bD_L \bRo_U$
of a symplectic matrix  $\bA\in\Sp$ implies a decomposition of the
associated \mfr\ $W_\bA $. \emph{If 
$U^tU\in\mathrm{U}(2d,\C)$ not  $d\times d$ block-diagonal, then there
exist an index $1\leq k\leq d$, symplectic matrices $\A,\B\in\Spp$ and a matrix $\Omega\in\mathrm{GL}(2d,\R)$ such that}
        \begin{equation}
            |\W_{\bA}(f,g)| = |\hat{\bD}_{\Omega} V^k_{\hat\B g} \hat\A f|,\qquad f,g\in\ltd.
          \end{equation}
 The following proof shows how to compute these matrices from $\bA $. 
\begin{proof} 
 To prove the claim, we successively decompose the occurring
 symplectic matrices and corresponding metaplectic operators. 
By assumption, for arbitrary  $\tau \in\T$, the matrix $\tau^2U^tU
= (\tau U)^t (\tau U)$ is not $d\times d$ block-diagonal. 
    By \cite[Lem.~4.3.]{GroechenigShafkulovska2024},  the imaginary
    part $B_\tau = \im(\tau U)$ of the matrix $\tau U = A_\tau+iB_\tau
    $ is invertible for all by finitely many $\tau =a+ib \in \T$. This
    means that  the corresponding symplectic matrix
    \begin{equation}
        \bRo_{\tau U} 
         = \begin{pmatrix}
            A_\tau &  B_\tau \\
            -B_\tau & A_\tau
        \end{pmatrix} \in \mathrm{Sp}(4d,\R) 
    \end{equation}
    has an invertible upper right $d\times d$ block $B_\tau $. Such a
    matrix is called \emph{free} and, by \cite[Prop.~62.]{Gosson2011},
    can be factored as 
    \begin{equation}\label{eq:free_decomp_tauU}
         \bRo_{\tau U} = \bV_{A_\tau B_\tau^{-1}} \bD_{B_\tau} \bJ \bV_{B_\tau^{-1}A_\tau}.
    \end{equation}
    By Lemma \cite[Lem.~5.3.]{GroechenigShafkulovska2024}, the
    symmetric matrix $P_\tau\coloneqq B_\tau^{-1}A_\tau$ is not
    $d\times d$ block-diagonal. 
    In the following, we fix $\tau$ such
    that $\tau U$ has an invertible imaginary part and drop the index
    $\tau $, i.e., we set $P \coloneqq P_\tau$. 
    
    We proceed with $\bJ\V_{B_\tau^{-1}A_\tau} = \bJ\V_{P}$. The proof of \cite[Thm.~3.4.]{GroechenigShafkulovska2024} relates the  metaplectic representation $\W_{\bJ\bV_{P}}(f,g) = \widehat{\bJ \bV_{P}}(f\otimes \overline{g})$ to a partial \stft. 
    We write $P = \begin{psmallmatrix}
        P_{11} & P_{12} \\ P_{12}^t & P_{22}
    \end{psmallmatrix}$ and 
    let $P_{12} = W_1 \Gamma W_2^t$ be the singular value decomposition of $P_{12}$, i.e., $W_1,W_2\in\mathrm{O}(d,\R)$ and $\Gamma =\diag(\Gamma_1,0_{(d-k)\times (d-k)})$ is a diagonal positive semi-definite matrix of rank 
    $k =\mathrm{rank}(\Gamma)=\mathrm{rank}(\Gamma_1)\geq 1$. We
    computed in \cite[Eq.~3.5.]{GroechenigShafkulovska2024} that
\begin{equation}\label{eq:citedLemma53_ov}
    \begin{split}
         |\det \Gamma_1|^{1/2}\F \hat\bV_{\begin{psmallmatrix}
             0 & \Gamma \\ \Gamma & 0
         \end{psmallmatrix}} (f\otimes \overline{g}) (\Gamma_1 x_1, x_2, \omega_1,\omega_2) =  V_{\hat{g}}^k h (\omega_1, x_2,x_1 , \omega_2),
    \end{split}
\end{equation}
where $h$ is the dilated partial Fourier transform~\footnote{If $d=k$, $h$ is just a dilation of $f$.}
\begin{equation}
   \begin{split}
        h(\Gamma _1 x_1, \omega_2)  & = |\det \Gamma_1|^{-1/2} \int_{\R^{d-k}} f_1 (x_1,x_2)e^{-2\pi i \omega_2\cdot x_2}\, dx_2 \\
        & = |\det \Gamma_1|^{-1/2} \hat\Ro_{\diag(I_k, i I_{d-k})} f(x_1,\omega_2).
   \end{split}
\end{equation}
Equivalently,
\begin{equation}
    h = \hat\D_{\diag(\Gamma_1,I_{d-k})} \hat\Ro_{\diag(I_k, i I_{d-k})} f.
\end{equation}
To capture the dilations on both sides of \eqref{eq:citedLemma53_ov}, we introduce the permutation matrix $\Pi\in\mathrm{O}(2d,\R)$ that acts by 
\begin{equation}
\Pi (\omega_1, x_2,x_1, \omega_2) = (x_1, x_2,\omega_1, \omega_2),\quad x_1,\omega_1\in\Rk,\, x_2,\omega_2\in\R^{d-k}.
\end{equation}
Then \eqref{eq:citedLemma53_ov} is equivalent to 
\begin{equation}\label{eq:citedLemma53_equiv}
    \begin{split}
         \F \hat\bV_{\begin{psmallmatrix}
             0 & \Gamma \\ \Gamma & 0
         \end{psmallmatrix}} (f\otimes \overline{g}) 
         =  
         \hat \bD_{\diag(\Gamma_1, I_{2d-k})}
         \hat\bD_{\Pi}
         V_{\hat{g}}^k \big(\hat\D_{\diag(\Gamma_1,I_{d-k})}
      \hat\Ro_{\diag(I_k, i I_{d-k})} f \big) .
    \end{split}
\end{equation} 
By \cite{GroechenigShafkulovska2024} (proof of Thm.~3.4),
\begin{equation}\label{eq:P-to-Gamma}
    \bJ \bV_P =  \bD_{\diag(W_1,W_2)}\bJ \bV_{\begin{psmallmatrix}
        0 & \Gamma \\ \Gamma & 0 
    \end{psmallmatrix}} \bD_{\diag(W_1^t, W_2^t)} \bV_{\diag(P_{11}, P_{22})},
\end{equation}
thus
\begin{equation}
    \begin{split}
        |\W_\bA(f,g)| 
        & \ =\  |\hat\bD_L \hat\bRo_{\tau U}\hat\bRo_{\overline{\tau} I}(f\otimes\overline{g})| \\
        & \overset{\eqref{eq:free_decomp_tauU}}{=} |\hat\bD_L \hat\bV_{A_\tau B_\tau^{-1}} \hat\bD_{B_\tau} \hat\bJ \hat\bV_{P} (\hat\Ro_{\overline{\tau} I}f\otimes\overline{ \hat\Ro_{{\tau} I} g})| \\
        & \overset{\eqref{eq:P-to-Gamma}}{=} |\hat\bV_{L^{-t}A_\tau B_\tau^{-1} L^{-1}} \hat\bD_L \hat\bD_{B_\tau} \hat\bD_{\diag(W_1,W_2)}\\
        &\qquad \qquad 
        \F \hat\bV_{\begin{psmallmatrix}
        0 & \Gamma \\ \Gamma & 0 
    \end{psmallmatrix}} \hat\bD_{\diag(W_1^t, W_2^t)} \hat\bV_{\diag(P_{11}, P_{22})} (\hat\Ro_{\overline{\tau} I}f\otimes\overline{ \hat\Ro_{{\tau} I} g})| \\
    &\ =\ |\hat\bD_{ LB_\tau \diag(W_1,W_2)}\\
        &\qquad \qquad 
        \F \hat\bV_{\begin{psmallmatrix}
        0 & \Gamma \\ \Gamma & 0 
    \end{psmallmatrix}} (\hat\D_{W_1^t}\hat\V_{P_{11}} \hat\Ro_{\overline{\tau} I}f\otimes
    \overline{ \hat\D_{W_2^t}\hat\V_{P_{-22}} \hat\Ro_{{\tau} I} g})| \\
    & \overset{\eqref{eq:citedLemma53_equiv}}{=} |\hat\bD_{ LB_\tau \diag(W_1,W_2)\,\diag(\Gamma_1, I_{2d-k}) \,\Pi }\\
        &\qquad \qquad 
         V^k_{\F  \hat\D_{W_2^t}\hat\V_{P_{-22}} \hat\Ro_{{\tau} I} g} \big(\hat\D_{\diag(\Gamma_1,I_{d-k})} \hat\Ro_{\diag(I_k, i I_{d-k})} \hat\D_{W_1^t}\hat\V_{P_{11}} \hat\Ro_{\overline{\tau} I}f \big) |.
    \end{split}
\end{equation}
Now it is clear that the claim is satisfied for 
\begin{equation}
    \begin{split}
        \Omega & = LB_\tau \diag(W_1,W_2)\,\diag(\Gamma_1, I_{2d-k})\,\Pi ,\\
        \A & = \hat\D_{\diag(\Gamma_1,I_{d-k})} \hat\Ro_{\diag(I_k, i I_{d-k})} \hat\D_{W_1^t}\hat\V_{P_{11}} \hat\Ro_{\overline{\tau} I},\\
        \B & =\F  \hat\D_{W_2^t}\hat\V_{P_{-22}} \hat\Ro_{{\tau} I}.
    \end{split} 
\end{equation}
\end{proof}


\begin{thebibliography}{10}

\bibitem{AmreinBerthier1977}
{\sc Amrein, W.~O., and Berthier, A.~M.}
\newblock On support properties of {$L\sp{p}$}-functions and their {F}ourier
  transforms.
\newblock {\em J. Functional Analysis 24}, 3 (1977), 258--267.


\bibitem{Benedicks1985}
{\sc Benedicks, M.}
\newblock On {F}ourier transforms of functions supported on sets of finite
  {L}ebesgue measure.
\newblock {\em J. Math. Anal. Appl. 106}, 1 (1985), 180--183.

\bibitem{BeurlingCollected1989}
{\sc Beurling, A.}
\newblock {\em The collected works of {Arne} {Beurling}. {Volume} 1: {Complex}
  analysis. {Volume} 2: {Harmonic} analysis. {Ed}. by {Lennart} {Carleson},
  {Paul} {Malliavin}, {John} {Neuberger}, {John} {Wermer}}.
\newblock Contemp. Mathematicians. Birkhauser, 1989.

\bibitem{BonamiEtAl2003}
{\sc Bonami, A., Demange, B., and Jaming, P.}
\newblock Hermite functions and uncertainty principles for the {F}ourier and
  the windowed {F}ourier transforms.
\newblock {\em Rev. Mat. Iberoamericana 19}, 1 (2003), 23--55.

\bibitem{CorderoGiacchi2024}
{\sc Cordero, E., and Giacchi, G.}
\newblock Metaplectic {G}abor frames and symplectic analysis of time-frequency
  spaces.
\newblock {\em Appl. Comput. Harmon. Anal. 68\/} (2024), Paper No. 101594, 31.

\bibitem{CorderoEtAl2022}
{\sc Cordero, E., Giacchi, G., and Rodino, L.}
\newblock {Wigner Analysis of Operators. Part {II}: Schr\"odinger equations}.
\newblock {\em Comm. Math. Phys. 405} (2024), no. 7, Paper No. 156, 39 pp.

\bibitem{CorderoEtAl2024}
{\sc Cordero, E., Giacchi, G., Rodino, L., and Valenzano, M.}
\newblock {Wigner Analysis of Fourier Integral Operators with symbols in the
  Shubin classes}.
\newblock  {\em NoDEA Nonlinear Differential Equations Appl. 31 } (2024), no. 4, Paper No. 69, 21 pp. 

\bibitem{corderogm25}
  {\sc E.~Cordero, G.~Giacchi, and E.~Malinnikova.}
\newblock Hardy's uncertainty principle for Schr\"odinger equations with
  quadratic Hamiltonians.
\newblock {\em Preprint, arXiv:2410.13818v1}, 2025.

\bibitem{CorderoRodino2022}
{\sc Cordero, E., and Rodino, L.}
\newblock Wigner analysis of operators. {P}art {I}: {P}seudodifferential
  operators and wave fronts.
\newblock {\em Appl. Comput. Harmon. Anal. 58\/} (2022), 85--123.

\bibitem{CorderoRodino2023}
{\sc Cordero, E., and Rodino, L.}
\newblock Characterization of modulation spaces by symplectic representations
  and applications to {S}chr\"{o}dinger equations.
\newblock {\em J. Funct. Anal. 284}, 9 (2023), Paper No. 109892, 40.


\bibitem{deGossonLuef2009}
{\sc de~Gosson, M., and Luef, F.}
\newblock Symplectic capacities and the geometry of uncertainty: the irruption
  of symplectic topology in classical and quantum mechanics.
\newblock {\em Phys. Rep. 484}, 5 (2009), 131--179.

\bibitem{Gosson2011}
{\sc de~Gosson, M.~A.}
\newblock {\em {Symplectic Methods in Harmonic Analysis and in Mathematical
  Physics}}.
\newblock Springer Basel, 2011.

\bibitem{gosson25}
N.~C. Dias, M.~de~Gosson, and J.~a.~N. Prata.
\newblock A metaplectic perspective of uncertainty principles in the linear
  canonical transform domain.
\newblock {\em J. Funct. Anal.}, 287(4), 2024.

\bibitem{Demange2005}
{\sc Demange, B.}
\newblock Uncertainty principles for the ambiguity function.
\newblock {\em J. London Math. Soc. (2) 72}, 3 (2005), 717--730.

\bibitem{Folland1989}
{\sc Folland, G.~B.}
\newblock {\em Harmonic analysis in phase space}, vol.~122 of {\em Annals of
  Mathematics Studies}.
\newblock Princeton University Press, Princeton, NJ, 1989.

\bibitem{FuerRzeszotnik2018}
{\sc F\"{u}hr, H., and Rzeszotnik, Z.}
\newblock A note on factoring unitary matrices.
\newblock {\em Linear Algebra Appl. 547\/} (2018), 32--44.

\bibitem{gelfand-vilenkin}
{\sc Gel{'}fand, I.~M., and Vilenkin, N.~Y.}
\newblock {\em Generalized functions. {V}ol. 4}.
\newblock Academic Press [Harcourt Brace Jovanovich Publishers], New York, 1964
  [1977].
\newblock Applications of harmonic analysis. Translated from the Russian by
  Amiel Feinstein.

  
\bibitem{GroechenigShafkulovska2024}
{\sc Gr{\"o}chenig, K., and Shafkulovska, I.}
\newblock Benedicks-type uncertainty principle for metaplectic time-frequency
  representations.
\newblock {\em arXiv:2405.12112\/} (2024).

\bibitem{GS25}
{\sc Gr\"ochenig, K., and Shafkulovska, I.}
\newblock A characterization of metaplectic time-frequency representations.
\newblock Preprint.

\bibitem{GroechenigZimmermann2001}
{\sc Gr\"{o}chenig, K., and Zimmermann, G.}
\newblock Hardy's theorem and the short-time {F}ourier transform of {S}chwartz
  functions.
\newblock {\em J. London Math. Soc. (2) 63}, 1 (2001), 205--214.

\bibitem{Hardy1933}
{\sc Hardy, G.~H.}
\newblock A {T}heorem {C}oncerning {F}ourier {T}ransforms.
\newblock {\em J. London Math. Soc. 8}, 3 (1933), 227--231.

\bibitem{havin-joricke}
{\sc Havin, V., and J{\"o}ricke, B.}
\newblock {\em The uncertainty principle in harmonic analysis}.
\newblock Springer-Verlag, Berlin, 1994.

\bibitem{hormander91}
{\sc H\"ormander, L.}
\newblock A uniqueness theorem of {B}eurling for {F}ourier transform pairs.
\newblock {\em Ark. Mat. 29}, 2 (1991), 237--240.

\bibitem{Jaming2007}
{\sc Jaming, P.}
\newblock Nazarov's uncertainty principles in higher dimension.
\newblock {\em J. Approx. Theory 149}, 1 (2007), 30--41.

\bibitem{Jaming2022}
{\sc Jaming, P.}
\newblock A simple observation on the uncertainty principle for the fractional
  {F}ourier transform.
\newblock {\em J. Fourier Anal. Appl. 28}, 3 (2022), Paper No. 51, 8.

\bibitem{janssen98}
{\sc Janssen, A. J. E.~M.}
\newblock Proof of a conjecture on the supports of {W}igner distributions.
\newblock {\em J. Fourier Anal. Appl. 4}, 6 (1998), 723--726.

\bibitem{Namias1980}
{\sc Namias, V.}
\newblock The fractional order {F}ourier transform and its application to
  quantum mechanics.
\newblock {\em J. Inst. Math. Appl. 25}, 3 (1980), 241--265.

\bibitem{Nazarov1993}
{\sc Nazarov, F.~L.}
\newblock Local estimates for exponential polynomials and their applications to
  inequalities of the uncertainty principle type.
\newblock {\em Algebra i Analiz 5}, 4 (1993), 3--66.

\bibitem{OzaktasKutay2001}
{\sc Ozaktas, H.~M., and Kutay, M.~A.}
\newblock The {F}ractional {F}ourier transform.
\newblock In {\em 2001 European Control Conference (ECC)\/} (2001), IEEE,
  pp.~1477--1483.

\bibitem{terMorscheOonincx2002O}
{\sc ter Morsche, H.~G., and Oonincx, P.~J.}
\newblock {On the Integral Representations for Metaplectic Operators}.
\newblock {\em J. Fourier Anal. Appl. 8} (2002), 245--258.

\bibitem{Wiener1929}
{\sc Wiener, N.}
\newblock {H}ermitian polynomials and {F}ourier analysis.
\newblock {\em J. Math. and Phys. 8}, (1929), 70--73.

\bibitem{Wilczock98}
{\sc Wilczock, E.}
\newblock Zur {F}unktionalanalysis der {W}avelet- und {G}abortransformation.
\newblock Thesis, TU M{\"u}nchen, 1998.

\end{thebibliography}
\end{document}